\documentclass[10pt]{amsart}

\usepackage{amsmath}
\usepackage{amsthm}
\usepackage{hyperref}
\usepackage{ragged2e}
\usepackage{enumitem}
\usepackage{tikz-cd}
\usepackage{mathtools}

\usepackage[
backend=biber,
style=alphabetic,
]{biblatex}

\usepackage{amssymb}
\newcommand{\bb}{\mathbb}

\newcommand{\Gal}{\operatorname{Gal}}
\newcommand{\Lap}{\triangle}
\newcommand{\divv}{\operatorname{div}}

\newcommand{\an}{\operatorname{an}}
\newcommand{\vphi}{\varphi}

\newcommand{\Res}{\operatorname{Res}}

\newcommand{\ovl}{\overline}

\newcommand{\Per}{\operatorname{Per}}
\newcommand{\cal}{\mathcal}

\newcommand{\Prep}{\operatorname{Prep}}
\newcommand{\vol}{\operatorname{vol}}

\newcommand{\FS}{\operatorname{FS}}
\newcommand{\eps}{\varepsilon}

\newcommand{\PGL}{\operatorname{PGL}}
\newcommand{\id}{\operatorname{id}}

\newcommand{\del}{\partial}
\newcommand{\h}{\widehat{h}}
	
\newtheorem{theorem}{Theorem}
\numberwithin{theorem}{section}
\newtheorem{lemma}[theorem]{Lemma}
\newtheorem{proposition}[theorem]{Proposition}

\numberwithin{example}{section}

\numberwithin{definition}{section}
\newtheorem{corollary}[theorem]{Corollary}

\makeindex
\title{Quantitative Equidistribution of Small Points for Canonical Heights}
\author{Jit Wu Yap}
\begin{document}

\begin{abstract}
Let $X$ be a smooth projective variety defined over a number field $K$ and let $\vphi: X \to X$ a polarized endomorphism of degree $d \geq 2$. Let $\h_{\vphi}$ be the canonical height associated to $\vphi$ on $X(\ovl{K})$. Given a generic sequence of points $(x_n)$ with $\h_{\vphi}(x_n) \to 0$ and a place $v \in M_K$, Yuan \cite{Yua08} has shown that the conjugates of $x_n$ equidistribute to the canonical measure $\mu_{\vphi,v}$. When $v$ is archimedean, we will prove a quantitative version of Yuan's result. We give two applications of our result to polarized endomorphisms $\vphi$ of smooth projective surfaces that are defined over a number field $K$. The first is an exponential rate of convergence for periodic points of period $n$ to the equilibrium measure and the second is an exponential lower bound on the degree of the extension containing all periodic points of period $n$. When $X$ is an abelian variety, we also give an upper bound on the smallest degree of a hypersurface that contains all points $x \in X(\ovl{K})$ satisfying 
$$[K(x):K] \leq D \text { and } \h_X(x) \leq \frac{c}{D^8}$$
for some fixed constant $c > 0$ where $\widehat{h}_X$ is the Neron--Tate height for $X$. 
\end{abstract}

\maketitle

\section{Introduction}
Let $K$ be a number field and $X$ a smooth projective variety over $K$. An endomorphism $\vphi: X \to X$ is polarized of degree $d$ if there exists an ample line bundle $L$ and an isomorphism $\vphi^*L \simeq L^{\otimes d}$ for an integer $d \geq 2$. For example, any endomorphism of $\bb{P}^r$ with degree $d \geq 2$ is polarizable of the same degree. Another example is the multiplication by $n$ map on any abelian variety $A$ for an integer $n \geq 2$, which is polarizable of degree $n^2$. By results of Fakkrhudin \cite{Fak14}, if $\vphi: X \to X$ is a polarized endomorphism then we may find an embedding $X \xhookrightarrow{} \bb{P}^N$ into projective space such that $\vphi$ extends to a morphism $\tilde{\vphi}: \bb{P}^N \to \bb{P}^N$.
\par 
If $h_L$ denotes a Weil height associated to the ample line bundle $L$, we can define a canonical height by the formula
$$\h_{\vphi}(x) = \lim_{n \to \infty} \frac{h_L(\vphi^n(x))}{d^n}$$
for all $x \in \ovl{K}$. By \cite{CS93}, this limit is well-defined and $\h_{\vphi}(x) = 0$ occurs exactly when $x$ is a preperiodic point for $\vphi$, i.e. there exists $m > n$ such that $\vphi^m(x) = \vphi^n(x)$. 
\par 
Points of small canonical height play an important role in diophantine geometry. For example the Bogomolov conjecture, proven by Ullmo \cite{Ulm98} and generalized by Zhang \cite{Zha98}, predicts that points of small Neron--Tate height on a subvariety $X$ of an abelian variety $A$ is not Zariski dense unless $X$ is a torsion coset of $A$. Due to the many connections between arithmetic dynamics and diophantine geometry, they also show up in important conjectures from arithmetic dynamics, such as the dynamical Bogomolov conjecture (see \cite{Zha06, GTZ11}) and the dynamical Andre--Oort conjecture \cite{BD14}. 
\par 
One of the main tools to deal with problems involving small points is equidistribution. For the case of abelian varieties, this goes back to Spziro--Ullmo--Zhang \cite{SUZ97}. It was subsequently proven for multiplicative tori by Bilu \cite{Bil97} and then for all polarized dynamical systems on projective varieties by Yuan \cite{Yua08}. In fact, Yuan's theorem holds more generally for heights coming from adelic line bundles. For simplicity, we will state Yuan's theorem only for canonical heights on $X$. 
\par 
Given a sequence of distinct points $(x_n)_{n \geq 1}$ of $X(\ovl{K})$, we say that the sequence is small if $\h_{\vphi}(x_n) \to 0$ and that the sequence is generic if no infinite subsequence of $(x_n)$ is contained in a proper subvariety $Z$ of $X$. Let $v: K \xrightarrow{} \bb{C}$ be an embedding. We may then view $\vphi$ as an endomorphism of $X(\bb{C}_v)$ and let $\mu_{\vphi,v}$ be the equilibrium measure of $\vphi$ on $X(\bb{C}_v)$ \cite{DS10}. We say that the Galois orbits of $(x_n)$ are equidistributed with respect to a measure $\mu$ if the discrete measures $\mu_n = \frac{1}{|F_{x_n}|} \sum_{y \in F_{x_n}} \delta_y$, where $F_{x_n}$ is the $\Gal(\ovl{K}/K)$-orbit of $x_n$, converge weakly to $\mu$.

\begin{theorem}[Theorem 3.7 \cite{Yua08}] \label{IntroTheoremYuan1}
Let $(x_n)$ be a generic and small sequence for the canonical height $\h_{\vphi}$. Then for any embedding $v \xrightarrow{} \bb{C}$, the Galois orbits of $(x_n)$ are equidistributed with respect to the measure $\mu_{\vphi,v}$ on $X(\bb{C}_v)$.
\end{theorem}

It is easy to see that the genericity of $(x_n)$ is a necessary condition. For example if $V$ is a preperiodic subvariety of $\bb{P}^r$ for $V$, then $V$ has infinitely many preperiodic points \cite{Fak03} and choosing our sequence $(x_n)$ from them, any possible limit measure must be supported on $V(\bb{C}_v)$. However, the equilibrium measure $\mu_{\vphi,v}$ is not supported on any subvariety as it has zero measure on pluripolar sets \cite[Proposition 1.2.3]{DS10}. 
\par 
Yuan's theorem holds too for non-archimedean places $v$ but one has to consider measures on the Berkovich analytification $X_v^{\an}$. The results of \cite{Yua08} has been extended by Yuan--Zhang \cite{YZ24} to the setting of heights arising from adelic line bundles on quasi-projective varieties (see also \cite{Gau23}).  
\par 
In dimension $1$, Favre--Rivera-Letelier \cite{FRL06} has proved a quantitative equidistribution theorem for heights arising from adelic measures on $\bb{P}^1$ using different ideas from \cite{SUZ97} and \cite{Yua08}. For the Neron--Tate height on elliptic curves, Baker and Petsche \cite{BP05, Pet09} have also proven quantitative bounds on the distribution of small points. However, their methods have yet to be generalized to higher dimensions and to date, the only known proof of \cite{Yua08} is following the strategy in \cite{SUZ97}. 
\par 
Our main result in this paper is a quantitative version of Theorem \ref{IntroTheoremYuan1} when $X$ is assumed to be smooth. Other than the case of $\bb{G}_m^n$ by D'Andrea--Narv\'{a}ez-Clauss--Sombra\cite{DNS17} which reduces to the one-dimensional case, we believe that this is the first result regarding quantitative equidistribution of small points that deal with endomorphisms on varieties of dimension larger than one. 
\par  
Given an ample line bundle $L$ on $X$, we fix some power of it such that $X$ embeds into some $\bb{P}^N$ under this power. We say that a subvariety $V$ of $X$ has degree $D$ with respect to $L$ if $V$ has degree $D$ inside $\bb{P}^N$.

\begin{theorem} \label{IntroTheorem1}
Let $X$ be a smooth projective variety defined over a number field $K$ and let $\vphi: X \to X$ a polarized endomorphism of degree $d \geq 2$, i.e. there exists an ample line bundle $L$ such that $\vphi^*L \simeq L^{\otimes d}$. Then there exists constants $c_1 = c_1(K,\vphi), c_2 = c_2(\vphi), c_3 = c_3(K,\vphi)$ such that for any $\eps > 0$ and embedding $v: K \xhookrightarrow{} \bb{C}$, if $f: X(\bb{C}_v) \to \bb{R}$ is a smooth function, then
\begin{enumerate}
\item there exists a hypersurface $H(f,\eps)$ defined over $K$ which has degree at most $ c_1 \eps^{-c_2}$ with respect to $L$,
\item for all $x \in X(\ovl{K})$ with $\Gal(\ovl{K}/K)$-orbit $F_x$, if $x \not \in H(f,\eps)$ then

$$\left|\frac{1}{|F_x|} \sum_{y \in F_x} f(y) - \int f(y) d \mu_{\vphi,v}(y) \right| \leq c_{f,3} \left(\frac{\h_{\vphi}(x)}{\eps} + c_3 \eps \right),$$
where $c_{f,3}$ is a constant depending on the derivatives of $f$ up to order $3$. 
\end{enumerate}
\end{theorem}

The constant $c_{f,3}$ depends on the $C^3$-norm of $f$. We will give the precise definition of $c_{f,3}$ later in Section \ref{sec: Background}. The constants $c_1,c_2,c_3, c_{f,3}$ can be made explicit in terms of the data given. For example, the exponent $c_2$ depends only on the dimension of $X$ along with a Holder exponent for the Green's function at the archimedean place $v$. In fact for abelian varieties, due to the smoothness of the metric, we show that we can take $c_2 = 8$, a constant independent of $X$. 

\begin{theorem} \label{IntroAbelianQuant1}
Let $X$ be an abelian variety over a number field $K$ and $L$ a symmetric ample line bundle on $X$. Let $\h_X(x)$ be the Neron--Tate height associated to $L$. Then there exists constants $c_1 = c_1(K,L), c_3=  c_3(K,L)$ such that for any $\eps > 0$ and embedding $v: K \xhookrightarrow{} \bb{C}$, if $f: X(\bb{C}_v) \to \bb{R}$ is a smooth function, then
\begin{enumerate}
\item there exists a hypersurface $H(f,\eps)$ defined over $K$ which has degree at most $ c_1 \eps^{-8}$ with respect to $L$,
\item for all $x \in X(\ovl{K})$ with $\Gal(\ovl{K}/K)$-orbit $F_x$, if $x \not \in H(f,\eps)$ then

$$\left|\frac{1}{|F_x|} \sum_{y \in F_x} f(y) - \int f(y) d \mu_{X,v}(y) \right| \leq c_{f,3} \left(\frac{\h_{X}(x)}{\eps} + c_3 \eps \right),$$
where $c_{f,3}$ is a constant depending on the derivatives of $f$ up to order $3$ and $\mu_{X,v}$ is the Haar measure for $A$ associated to $v$.  
\end{enumerate}
\end{theorem}

The condition of $x \not \in H(f,\eps)$ appears in Theorem \ref{IntroTheoremYuan1} through the genericity of the sequence $(x_n)$. Although it is not clear from the statement of Theorem \ref{IntroTheoremYuan1} that such a hypersurface $H(f,\eps)$ exists, it  follows immediately from the proof of Theorem \ref{IntroTheoremYuan1}, for e.g. see \cite[Lemma 23]{Kuh21}. The new result in Theorem \ref{IntroTheorem1} is a bound on the degree of $H(f,\eps)$, which tells you how ``generic" $x$ has to be to achieve a desired accuracy for a given test function $f$ and $\eps > 0$. 
\par 
Recently, there have been many breakthrough involving uniform bounds of unlikely intersection problems with proofs relying on equidistribution, e.g. \cite{Kuh21, GGK21, MS22, GTV23, DM24, DMY24}. These results prove uniform bounds on the number of special points as one varies over a family, but it is not known how these bounds depend on the parameters of the given family as Theorem \ref{IntroTheoremYuan1} and its variants are all inexplicit. We hope that Theorem \ref{IntroTheorem1} would be the first step in producing explciit constants for the above results.
\par 
We now give two applications of Theorem \ref{IntroTheorem1}. The first concerns the rate of equidistribution of periodic points to the equilibrium measure. Let $\vphi: \bb{P}^r \to \bb{P}^r$ be an endomorphism of degree $d \geq 2$ defined over $\bb{C}$ and let $\Per_n$ denote the set of periodic points of period $n$ for $\vphi$, where we do not count with multiplicity. Then Briend--Duval \cite{BD01} has shown that as $n$ goes to infinity, we have 
\begin{equation} \label{eq: Periodic1}
\frac{1}{|\Per_n|}  \sum_{x \in \Per_n} \delta_x \to \mu_{\vphi}.
\end{equation}
Dinh--Sibony \cite[Problem 7]{DS17} asks for an estimate of the rate of convergence of \eqref{eq: Periodic1}. When $\vphi$ is defined over a number field $K$, we will show that the convergence is of an exponential rate. Over $\bb{C}$, this has recently been shown for $\bb{P}^1$ by Dinh--Kaufman \cite{DK24}. 

\begin{theorem} \label{IntroTheoremPeriodicRate1}
Let $X$ be a smooth projective surface over a number field $K$ and $\vphi: X \to X$ a polarized endomorphism of degree $d \geq 2$. Then there exists constants $C > 0$ and $\lambda > 1$, depending on $\vphi$, such that if $v$ is an archimedean place of $K$, then for any smooth function $f: X(\bb{C}_v) \to \bb{R}$, we have
$$\left|\frac{1}{|\Per_n|} \sum_{x \in \Per_n} f(x) - \int f d \mu_{\vphi,v} \right| \leq C c_{f,3}\lambda^{-n}.$$
\end{theorem}

Our second application of Theorem \ref{IntroTheorem1} concerns the set of preperiodic points $x$ of $\vphi$ that satisfy $[K(x):K] \leq D$ for some fixed positive integer $D \geq 1$. In \cite{Bak06}, Baker showed that if $\vphi: \bb{P}^1 \to \bb{P}^1$ is a rational map defined over $K$ of degree $d \geq 2$, then there are at most $O_{\vphi}(D^{1+\eps})$ preperiodic points in a fixed extension $L/K$ of degree $D$. Our first application is a higher dimensional analogue of Baker's result, where we control the geometry of preperiodic points of bounded degree $\leq D$. We will let the set of preperiodic points having degree $\leq D$ over $K$ be denoted by $\Prep_{\vphi,D}$ and similarly denote the set of periodic points by $\Per_{\vphi,D}$.  

\begin{theorem} \label{IntroTheorem2}
Let $X$ be a smooth projective variety defined over a number field $K$ and $\vphi: X \to X$ a polarized endomoprhism of degree $d \geq 2$. There exists a constant $e = e(\vphi) > 0$ such that for any positive integer $D > 0$, there exists a hypersurface $H$ of degree at most $O_{\vphi,K}(D^e)$, such that $\Prep_{\vphi,D} \subseteq H$.  
\end{theorem}

One difference from Theorem \ref{IntroTheorem2} and Baker's result is that we are able to handle all extensions of degree $\leq D$ at once, instead of fixing a single extension $L$. In the case of $r = 1$, our hypersurface $H$ is an union of at most $O_{\vphi}(D^e)$ points and so we obtain a bound on $|\Prep_{\vphi,D}|$. This appears to already be a new result, although we will show that one can also deduce this from Favre--Rivera-Letelier's quantitative equidistribution theorem (see Proposition \ref{GaloisDimOne1}). 
\par 
In the case of abelian varieties, due to the stronger estimates from Theorem \ref{IntroAbelianQuant1}, we can deduce the following result.

\begin{theorem} \label{IntroAbelianGalois1}
Let $X$ be an abelian variety over a number field $K$ and let $\h_X(x)$ denote the Neron--Tate height of $X$ with respect to a symmetric ample line bundle $L$. Then there exists $c = c(L,K) > 0$, such that if  
$$S = \{ x \in X(\ovl{K}) \mid \h_{X}(x) \leq \frac{c}{D^8} \text{ and } [K(x):K] \leq D\},$$
then $S$ is contained in a hypersurface of degree $O_{L,K}(D^{33})$.
\end{theorem}

In \cite{Mas84}, Masser proved an upper bound on the number of small points for an abelian variety $X$ that are defined over a fixed extension of degree $D$. More recently, Gaudron--Remond \cite{GR22} proves an explicit version of Masser's result through different methods. Again, our theorem differs from both of theirs as it is able to handle all extensions of degree $\leq D$ at once rather than just a fixed extension.
\par 
Applying Theorem \ref{IntroTheorem2} to $\Per_n$, the set of periodic points of period $n$, gives us a generalization of \cite[Corollary 1.19]{Bak06}. Let $K_n = K(\Per_n)$ denote the field obtained by adjoining every element of $\Per_n$ to $K$. In the case of $\vphi: \bb{P}^1 \to \bb{P}^1$, Baker showed that $[K_n:K] \geq O(c^n)$ for any $c < d$. We show that for polarized endomorphisms of projective surfaces $X$, we can find some constant $c > 1$ such that $[K_n:K] \geq c^n$.

\begin{theorem} \label{IntroTheorem5}
Let $X$ be a projective smooth surface defined over a number field $K$ and let $\vphi: X \to X$ be a polarized endomorphism of degree $d \geq 2$. Then there exists $c = c(\vphi,K)$ such that $[K_n:K] \geq c^n$. In fact there exists $c' > 1$ such that  for $1- O(c'^{-n})$ proportion of points $x \in \Per_n$, we will have $[K(x):K] \geq c^n$. 
\end{theorem}

When $X = \bb{P}^2$, under some conditions, we are also able to get a numerical bound on $|\Per_{\vphi,D}|$, the set of periodic points for $\vphi$ of degree $\leq D$.

\begin{theorem} \label{IntroTheorem3}
Let $\vphi: \bb{P}^2 \to \bb{P}^2$ be an endomorphism of degree $d \geq 2$ defined over a number field $K$. Further assume that there are only finitely many periodic curves for $\vphi$. Then there exists $e = e(\vphi) > 0$ such that $|\Per_{\vphi,D}| \leq O_{\vphi,K}(D^e)$.
\end{theorem}

Heuristically one expects that a general endomorphism $\vphi: \bb{P}^r \to \bb{P}^r$ should have only finitely many periodic subvarieties, although this is not known yet. However, results of Xie \cite{Xie24} show that this is true for a general endomorphism $\vphi: \bb{A}^2 \to \bb{A}^2$ that extends to an endomorphism of $\bb{P}^2$. In some cases where we have infinitely many periodic subvarieties, it might still be possible to still deduce a numerical bound if we are able to understand their geometry. As an example, check out Theorem \ref{Galois4}.
\par 
We remark that there are previous known bounds on the number of periodic points in a fixed extension $L/K$ of degree $D$ for an endomorphism $\vphi: X \to X$ given by Hutz \cite{Hut09, Hut18} and Huang \cite{Kep21}, following the ideas of Morton--Silverman \cite{MS94}. The bounds they get are uniform over certain families, but they grow super-exponentially large as $D$ increases. 

\subsection{Comparison with other results} We now compare Theorem \ref{IntroTheorem1} with other quantitative equidistribution theorems in different settings. In $\bb{P}^1$, equidistirbution of small points for canonical heights was proven independently by Baker--Rumely \cite{BR06, BR10}, Chambert-Loir \cite{CL06} and Favre--Rivera-Letelier \cite{FRL06}.  Favre--Rivera-Letelier's results also give an explicit control over the rate of equidistribution as follows: if $x$ is a point of $\bb{P}^1(\ovl{\bb{Q}})$ and we let $F_x$ denote the Galois orbit of $x$, then for any $C^2$ function $f$ and archimedean place $v \in M_K$, we have
\begin{equation} \label{eq: IntroIneq2}
\left|\frac{1}{|F_x|} \sum_{y \in F_x} f(y) - \int f d\mu_{\vphi,v} \right| \leq \h_{\vphi}(x) + \frac{c \log |F|}{|F|^{1/2}} (\Lap f, \Lap f)_v
\end{equation}
where $c = c(\vphi,K) > 0$ is some constant depending on our rational map $\vphi$ and the number field $K$ and $( \, , \, )_v$ is some energy pairing of measures. Hence if we wanted the LHS to be less than $\eps$, we require $\h_{\vphi}(x) \leq \eps$ and $|F|$ to be at least $\eps^{-2+\delta}$ in size for some fixed $\delta > 0$. For a more explicit version of what the constant $c$ is in terms of $\vphi$ and $K$, one may look at \cite{Fil17}. For elliptic curves over function fields, a similar statement is proven by Petsche \cite{Pet09}. 
\par 
When specialized to $X = \bb{P}^1$, Theorem \ref{IntroTheorem1} tells us that for each smooth function $f$ and $\eps > 0$, there is a finite $\Gal(\ovl{K}/K)$-invariant set of points $H(f,\eps)$ of size at most $c_2 \eps^{-c_1}$ such that if $x$ is not a member of $H(f,\eps)$, then
\begin{equation} \label{eq: IntroIneq1}
\left|\frac{1}{|F_x|} \sum_{y \in F_x} f(y) - \int f d\mu_{\vphi,v} \right| \leq \frac{\h_{\vphi}(x)}{\eps} + c_3 \eps
\end{equation}
where $c_2,c_3$ depends on $\vphi$ and $K$ and $c_1$ depends only on $\vphi$. Hence for example, if $|F_x|$ were larger than $c_2 \eps^{-c_1}$ in size, then $x$ cannot lie in $H(f,\eps)$ and so our \eqref{eq: IntroIneq1} would always hold. This gives a similar statement to Favre--Rivera-Letelier's as we only require $|F_x|$ to be larger than some polynomial in $\eps^{-1}$ to achieve an $\eps$ accuracy, but is much weaker in the dependencies. For example we have an $\eps$ in the denominator of $\h_{\vphi}(x)$ but \eqref{eq: IntroIneq2} does not and our exponent in $\eps$ depends on $\vphi$ whereas \eqref{eq: IntroIneq2} only requires an exponent of $2+\delta$. D'Andrea--Narv\'{a}ez-Clauss--Sombra's theorem for $\bb{G}_m^n$ reduces to \cite{FRL06} and thus their estimates have the same form as Favre--Rivera-Letelier's estimates.
\par 
We now move onto dynamical settings, where one considers all iterated preimages $f^{-n}(x)$ of a point $x \in \bb{P}^r(\bb{C})$. Here, both $f$ and $x$ need not be defined over a number field. Equidistribution of preimages for a rational map $f: \bb{P}^1 \to \bb{P}^1$ was first proven by Freire--Lopes--Mane \cite{FLM83} and Lyubich \cite{Lyu83} independently. The case of polynomials was proven by Brolin \cite{Bro65} earlier. In the higher dimensional setting, there have been many results concerning equidistribution of subvarieties or points, see for e.g. \cite{BD01, FJ03, DS03}. One may consult \cite{DS10} for an excellent survey on these results.
\par 
On the quantitative side, we have the following result of Dinh--Sibony \cite{DS10b}. For a holomorphic endomorphism $\vphi: \bb{P}^r \to \bb{P}^r$ of algebraic degree $d \geq 2$, there exists a maximal totally invariant analytic subset $\cal{E}$. For simplicity, we will state their result when $\cal{E} = \emptyset$. Under this assumption, there exists a $\lambda > 1$ such that for any $C^2$ function $f$ and any $a \in \bb{P}^r(\bb{C})$, we have
$$\left|\frac{1}{|F_n|} \sum_{y \in F_n} f(y) - \int f d \mu_{\vphi} \right| \leq \frac{A ||f||_{C^2}}{ \lambda^{n}}$$
where $F_n = \vphi^{-n}(a)$ as a multiset and $A$ is a constant independent of $n,f$ and $a$. If one knew the existence of $\lambda > 1$ such that no hypersurface of degree $\lambda^n$ contains $\frac{1}{\lambda^n}$-proportion of the points of $f^{-n}(a)$, we may apply Theorem \ref{IntroTheorem1} to obtain a similar result. However, it is not clear whether such a $\lambda$ exists. 
\par 
Dinh--Ma--Nguyen \cite{DMN17} have also proven a quantitative equidistribution theorem for Fekete points. By the work of Looper \cite{Loo24}, small points for the canonical height may be viewed as ``almost" Fekete points for the homogeneous filled Julia set and so there might be a way to modify Dinh--Ma--Nguyen's approach to work in our setting. In their theorem, the error bound depends on the $C^3$-norm of the test function $f$, similar to our error bound in Theorem \ref{IntroTheorem1}. This is not a coincidence as we both rely on quantitative estimates of asymptotic expansion of Bergman kernels, which is where the $C^3$-norm of $f$ appears.

\subsection{Proof strategy} We follow the approach of Szpiro--Ullmo--Zhang \cite{SUZ97} and Yuan \cite{Yua08}. The general idea is to consider an adelic line bundle $\ovl{L}$ on $X$ whose corresponding height function is equal to the canonical height $\h_{\vphi}$. Given a smooth function $f: X(\bb{C}_v) \to \bb{R}$ for an archimedean place $v$, we consider the adelic line bundle $\ovl{L}(\eps f)$ where we modify the metric at the place $v$ by multiplying with $e^{-\eps f}$. 
\par 
In the case where $X$ is an abelian variety and $\h_{\vphi}$ is the Neron--Tate height, the adelic line bundle $\ovl{L}$ is a positive line bundle at the archimedean places and so for $\eps$ small enough, the perturbed line bundle $\ovl{L}(\eps f)$ remains positive. This allows us to apply Zhang's inequality \cite{Zha92, Zha95} to $\ovl{L}(\eps f)$ and then some straightforward algebraic manipulation allows us to prove equidistribution.
\par 
For canonical heights in general, the adelic line bundle $\ovl{L}$ is not necessarily positive but only semipositive and hence even if $\eps$ is arbitrarily small, it is possible for $\ovl{L}(\eps f)$ to be not semipositive. To handle this issue, Yuan proves an arithmetic Siu's inequality which allows him to give a lower bound on the volume of $\ovl{L}(\eps f)$ in terms of $\ovl{L}$ and $f$. This then allows him to prove a version of Zhang's inequality for $\ovl{L}(\eps f)$.  
\par 
Zhang's inequality gives a lower bound on the height of points that lie outside a certain codimension one subvariety $H$, and it is this subvariety for $\ovl{L}(\eps f)$ that will be our hypersurface $H(f,\eps)$ in Theorem \ref{IntroTheorem1}. The hypersurface arises as the divisor of a section $s$ of $L^n$ for some large $n > 0$, where $L$ is the underlying line bundle of $\ovl{L}$, and thus to control the degree of $H(f,\eps)$, we have to control how large $n$ is. 
\par 
For simplicity, we will assume that $K = \bb{Q}$. Then $H^0(X,L^n)$ is a $\bb{Q}$-vector space and we may tensor up to $\bb{R}$ to obtain a $\bb{R}$-vector space $V_n$. The sections $s$ of $H^0(X,L^n)$ that satisfies $|s|_{\sup,v} \leq 1$ for all non-archimedean places $v$ forms a $\bb{Z}$-lattice $M_n$ inside $V_n$. On the other hand, we may look at all sections $s \in V_n$ that satisfies $|s|_{\sup,v} \leq 1$ for the archimedean place $v$ of $\bb{Q}$. This forms a convex ball $B_n$ inside $V_n$. Fixing any choice of Haar measure, the quantity
$$\chi(\ovl{L}^n) := \frac{1}{n} \left(\log \vol(B_n) - \log \vol(M_n) \right)$$
makes sense and is independent of our choice of Haar measure, where $\vol(M_n)$ denotes the volume of its fundamental region. To control the degree of $H(f,\eps)$, one then has to prove lower bounds of $\chi(\ovl{L}^n)$ for large $n$. We note that quantitative upper bounds have been proven by Xinyi Yuan and Tong Zhang \cite{YZ13}. 
\par 
In fact if $\ovl{L}$ is semipositive, as $n \to \infty$, there is a formula for the asymptotics of $\chi(\ovl{L}^n)$ given by the arithmetic Hilbert--Samuel theorem \cite[Theorem 1.4]{Zha95} in terms of the arithmetic self-intersection number of $\ovl{L}$. However, this relies on the deep results of Gillet--Soule \cite{GS92} and Bismut--Vasserot \cite{BV89} (see also Moriwaki's book \cite{Mor14} for an exposition), which makes it difficult to deduce an explicit lower bound for $\chi(\ovl{L}^n)$. 
\par 
We instead give an alternative approach to obtain an explicit lower bound for $\chi(\ovl{L}^n)$, instead of relying on the arithmetic Hilbert--Samuel theorem. The idea is to find a suitable ``canonical" basis $\cal{B}_n$ of $H^0(X,L^n)$ with respect to $\vphi$, such that if $M'_n$ denotes the lattice spanned by this basis, we can find a suitable $c > 0$ such that its fundamental domain is contained inside $cB_n$ and similarly the lattice $cM'_n$ is contained in $M_n$. Controlling $c$ in terms of $n$ then gives us a lower bound on $\chi(\ovl{L}^n)$. 
\par 
The idea of constructing such a canonical basis $\cal{B}_n$ goes back to Baker \cite{Bak06} in his work on proving lower bounds for average values of Arakelov-Green functions associated to $\vphi$ in the case of $\bb{P}^1$. This result, along with Baker's construction, has recently been generalized by Looper \cite{Loo24} to the setting of a general polarized dynamical system. Our construction of this basis, which might look different on first sight, is essentially the same as Looper's.  A key obstacle in generalizing Baker's work is that it is much harder to prove that a set of homogeneous polynomials form a basis when there are more than $2$ variables. Looper gets around this by instead constructing a canonical spanning set of polynomials which is much easier. Then there must exist a basis formed by a subset of this canonical set which turns out to be enough for applications. 
\par 
For equidistribution, we actually need a lower bound on $\chi(\ovl{L}(\eps f)^n)$ and not on $\chi(\ovl{L}^n)$. Using ideas from Yuan's proof of the arithmetic Siu's inequality which goes back to Abbes--Bouche \cite{AB95}, we are able to directly prove a lower bound of $\chi(\ovl{L}(\eps f)^n)$ of the form
$$\chi(\ovl{L}(\eps f)^n) \geq \chi(\ovl{L}^n) + \eps \int f d \mu_{\vphi,v} + O(\eps^2)$$
for $n$ large enough. Hence a lower bound on $\chi(\ovl{L}^n)$ leads to a lower bound on $\chi(\ovl{L}(\eps f)^n)$ which then allows us to prove Theorem \ref{IntroTheorem1}. The main input is a quantitative control on the asymptotic expansion of Bergman kernels in terms of the positivity of the underlying $(1,1)$-form $\omega$ of a positive line bundle $\ovl{L}$. 
\par 
We remark that there is recently a new proof of arithmetic Hilbert--Samuel via deformation to the normal cone by Dorian Ni \cite{Ni22}. Also, David--Phillipon \cite{DP99} prove lower bounds on $\chi(\ovl{L}^n)$ that are explicit, but are not optimal in the sense that they do not give the correct asymptotics. 

\subsection{Outline of Paper} We now give the outline of the paper. Section $2$ provides some complex geometry background along with how to regularize a continuous metric on a line bundle $L$ to obtain a smooth metric. Sections $3$ and $4$ contains quantitative estimates which are sufficient to prove our theorems for abelian varieties. Section $5$ contains the main global input coming from Looper's generalization of Baker's construction. Theorems \ref{IntroAbelianQuant1} and \ref{IntroAbelianGalois1} are then proven in Section $6$. We then move onto more technical quantitative estimates in Section $7$, which are needed in the general setting where our canonical metric is only semipositive. We prove our main theorem, Theorem \ref{IntroTheorem1}, in general in Section $8$. Theorem \ref{IntroTheoremPeriodicRate1} is proven in Section $9$ and Theorems \ref{IntroTheorem2},  \ref{IntroTheorem5}, \ref{IntroTheorem3} are proven in Section $10$.

\subsection{Notations} 
We now state some notation. Throughout the paper, we will use $O(X)$ to denote a quantity which is bounded from above by a constant multiplied by $X$. When the constant depends on some other variable $Y$ that are not fixed, we will try to make it clear by writing it as $O_Y(X)$ instead. 
\par 
In a local chart with coordinates $z_1,\ldots,z_n$, for a $(p,q)$-form 
$$u(z) = \sum_{|I|=p, |J|=q} u_{I,J}(z) dz_I \wedge d \ovl{z}_J,$$
we will let 
$$\del u = \sum_{|I|=p, |J|=q}  \sum_{k=1}^{n} \frac{\partial u_{I,J}}{\partial z} dz_k \wedge dz_I \wedge d \vol{z}_J$$
and similarly $\ovl{\del}$ will be the analogue with $\frac{\partial}{\partial \ovl{z}}$. We will let $d = \del + \ovl{\del}$ and $d^c = \frac{i}{\pi} (\del - \ovl{\del})$. 
\par 
If $\ovl{L}$ is a metrized line bundle and $f: X(\bb{C}) \to \bb{R}$ is a function, we will let $\ovl{L}(f)$ denote the metriezd line bundle with norm $| \cdot | e^{-f}$, where $| \cdot |$ is the original norm of $\ovl{L}$. When $f$ is a constant function $c$, we will simply denote it by $\ovl{L}(c)$. 

\subsection{Acknowledgements} The author would like to thank Tien--Cuong Dinh, Laura DeMarco, Nicole Looper, Niki Myrto Mavraki, Thomas Gauthier, Junyi Xie and Xinyi Yuan for helpful discussions and encouragement. The work was supported in part by the National Science Foundation.

\section{Background and Preliminaries} \label{sec: Background}

\subsection{Complex geometry background}
Let $X$ be a complex manifold of dimension $r$ and let $\cal{U}_I = \{U_{\alpha}\}_{\alpha \in I}$ be a choice of charts for $X$ where $I$ is a finite set. We say that $\cal{U}_I$ is a choice of good charts for $X$ if each $U_{\alpha}$ is biholomorphic to the open polydisc of radius $3$, $\bb{D}_3$, and every point $x \in X$ lies inside the unit polydisc $\bb{D}_1$ for one of these charts. 
\par 
Now fix a good choice of charts $\cal{U}_I$. Given a $C^k$ function $f: X \to \bb{R}$, viewing $X$ as a real manifold of dimension $2r$, we may define a $C^k$-norm $c_{f,k}$ as the sum of its order $\leq k$ derivatives in the polydisc $\bb{D}_2 \subseteq U_{\alpha}$. We will also assume that $c_{f,k} \geq 1$ by adding $1$ to this sum. In other words, if $\vphi_{\alpha}: U_{\alpha} \to \bb{D}_3$ is our biholomorphism, then we have 
$$c_{f,k} = 1 + \ \sum_{\alpha \in I} \sum_{i_1 + \cdots + i_{2r} \leq k}  \left(\sup_{|\vphi_{\alpha}(z)| \leq 2} \left| \frac{\partial^{i_1 + \cdots i_{2r}} (f \circ \vphi_{\alpha}^{-1})}{ \partial x_1^{i_1} \cdots \partial y_r^{i_{2r}}}(z)   \right| \right) .$$
Similarly, given a $C^k$ differential form $\omega$, we may define the constants $c_{\omega,k}$. We now prove some basic propositions that show these constants are comparable independent of the choice of good charts.

\begin{lemma} \label{GoodCharts1}
Let $\cal{U}_I$ and $\cal{V}_J$ be two different choice of good charts for $X$. Then for any $k \geq 0$, there exists $C = C(k, \cal{U}_I, \cal{V}_J)$, depending on $k$ and the choice of our good charts, such that for any $C^k$ function $f: X \to \bb{R}$, if $c_{f,k}$ is the norm defined via $\cal{U}_I$ and $c'_{f,k}$ is defined via $\cal{V}_J$, then
$$c_{f,k} \leq C c'_{f,k}.$$
\end{lemma}

\begin{proof}
Let $U_i$ be a chart of $\cal{U}_I$. Then it is isomorphic to $\bb{D}_3$, Given any $x \in \ovl{\bb{D}}_2$, there is a corresponding chart $V_j \simeq \bb{D}_3$ such that $x$ lies inside $\bb{D}_1$ of $V_j$. We may then find an open set around $x$, contained inside a compact subset of $U_i \cap V_j$ and also contained inside $\bb{D}_2$ of $V_j$. Then the $k^{th}$ derivatives of our transition function between $U_i$ and $V_j$ is bounded on this open set. 
\par 
Now since $\ovl{\bb{D}}_2$ is compact, we may cover it with finitely many such open sets. This then implies that we may upper bound the $k^{th}$ derivative of $f$ on $\bb{D}_2$ of $U_i$ by some constant multiplied by $c'_{f,k}$ as desired. We then do this for each $U_i$ of our chart $\ovl{U}_I$ to obtain that $c_{f,k} \leq C c'_{f,k}$ for some $C$ depending on $k,\cal{U}_I$ and $\cal{V}_J$. 
\end{proof}

It is also useful to understand how the norms change when restricting from a manifold $X$ to a submanifold $Y$. 

\begin{lemma} \label{GoodCharts2}
Let $X$ be a complex manifold and $Y \subseteq X$ be a complex submanifold. Let $\cal{U}_I$ be a choice of good charts for $X$ and $\cal{V}_J$ be a choice of good charts for $Y$. Then for any $k \geq 0$, there exists $C = C(k, \cal{U}_I, \cal{V}_J)$, depending on $k$ and the choice of our good charts, such that for any $C^k$ function $f: X \to \bb{R}$, if $c_{f,k}$ is the norm defined via $\cal{U}_I$ and $c'_{f,k}$ is the norm of $f\mid_Y: Y \to \bb{R}$ defined via $\cal{V}_J$, then
$$c'_{f,k} \leq C c_{f,k}.$$
\end{lemma}

\begin{proof}
By Lemma \ref{GoodCharts1}, it suffices to prove the existence of $C$ for a single choice of $\cal{U}_I$ and $\cal{V}_J$. We first choose $\cal{V}_J$ such that for each $V_j \in \cal{V}_J$, the submanifold $Y \cap V_j$ is given by $\{z_1 = z_2 = \cdots = z_k = 0\}$. To do so, for each $x \in X$ we may find a local chart $V_x$ such that $Y \cap V_x$ is given by $\{z_1 = \cdots = z_k = 0\}$. We then choose a polydisc $\bb{D}_r$ centered at $x$ that also lies inside $V_x$, and we let $V'_x$ be the polydisc $\bb{D}_{r/3}$ of one third radius. Then $\{V'_x\}_{x \in X}$ form an open cover of $X$ and by compactness we may choose a finite open cover $V'_{x_1},\ldots,V'_{x_n}$. We may then use $V_{x_1},\ldots,V_{x_n}$ as a choice of good charts. 
\par 
Given such a choice of good charts $\cal{V}_J$, we then set $U_j = Y \cap V_j$ and clearly $U_j$ is isomorphic to $\bb{D}_3$ of the correct dimension for $Y$. Then using $\cal{U}_J$ and $\cal{V}_J$, it is clear that we have $c_{f,k} \leq c'_{f,k}$ as desired.
\end{proof}

Lemmas \ref{GoodCharts1} and \ref{GoodCharts2} similarly hold for $c_{\omega,k}$. Let $L$ be a line bundle on $X$ equipped with a hermitian metric $e^h$. Since each $U_{\alpha}$ is biholomorphic to a polydisc, it follows that $L$ can be trivialized over each $U_{\alpha}$. We then fix a choice of non-vanishing local sections $t_{\alpha}$ for each $U_{\alpha}$, giving us the isomorphism $L |_{U_{\alpha}} \simeq U_{\alpha} \times \bb{C}$. Then we may find functions $h_{\alpha}$ such that $e^{h_{\alpha}}$ represents our hermitian metric. We may then similarly define the constants $c_{h,k}$ if our metric $h$ is $C^k$ and Lemmas \ref{GoodCharts1} and \ref{GoodCharts2} also hold.  
\par 
If $h$ is a continuous metric but not differentiable, we may still make sense of $dd^c h$ in the sense of distributions. In our setting, it will be the case that $dd^c h$ is a positive current, i.e. that $dd^c h \geq 0$. It will be necessary for us to approximate $h$ by a sequence of smooth metrics $h_{\eps}$, such that $dd^c h_{\eps}$ is almost positive. For a general compact manifold $X$, this has been done in a quantitative manner by Dinh--Ma--Nguyen \cite[Theorem 2.1]{DMN17} with a rather involved argument. As we only have to deal with the case of $X = \bb{P}^r$ due to our endomorphisms being polarized, we follow a simpler idea of Dinh--Sibony \cite{DS10b} to use the action of $\PGL_{r+1}(\bb{C})$ on $\bb{P}^r$ to regularize our metric $h$. 
\par 
We first explain how to regularize a continuous function $f: \bb{P}^r \to \bb{R}$. The group $\PGL_{r+1}(\bb{C})$ is a complex Lie group that acts transitively on $\bb{P}^r(\bb{C})$. We fix a holomorphic chart near the identity element $\id$ of $\PGL_{r+1}(\bb{C})$ isomorphic to the polydisc $\bb{D}_2$ and let $y$ be the holomorphic coordinate such that $y = 0$ at $\id$. We then let $\rho$ be a smooth probability measure with compact support in $|y| < 1$. We will assume that $\rho$ is radial (depends only on $|y|$) and that the involution $\tau \mapsto \tau^{-1}$ preserves the norm of $y$. 
\par 
Following Dinh--Sibony, for $\theta \in \bb{C}$ we define $h_{\theta}(y) = \theta y$ to be multiplication by $\theta$. For $|\theta| \leq 1$, we define $\rho_{\theta} = (h_{\theta})_* \rho$. This is smooth probability measure supported on $\{|y| \leq \theta\}$ which converges to the Dirac mass on $\id$ as $\theta \to 0$. We now define our regularization $f_{\theta}$ as 
$$f_{\theta}(x) = \int_{\PGL_{r+1}(\bb{C})} (f \circ \tau_y)(x) d \rho_{\theta}(y)$$
where $\tau_y$ is the automorphism corresponding to the coordinate $y$. By Dinh--Sibony, we have the following estimate.

\begin{proposition} \label{Regularization1}
Fix a positive integer $k$. Then there exists an uniform constant $c_{k} > 0$, depending only on $k$ and $r$, such that
$$c_{f_{\theta},k} \leq c_{k} c_{f,0} |\theta|^{-2r^2-4r-k}.$$
\end{proposition}

\begin{proof}
This is Proposition 2.1.6 of Dinh--Sibony \cite{DS10b}.     
\end{proof}

Fix the Fubini-Study metric on $\bb{P}^r(\bb{C})$. We now further assume that $f(x)$ is Holder continuous with respect to our metric. Then we also have the following control on $|f_{\theta} - f|$ as $\theta \to 0$. 

\begin{lemma} \label{Regularization2}
Assume there exists $C,\kappa > 0$ such that $|f(x)-f(y)| \leq C|x-y|^{\kappa}$ for all $x,y \in \bb{P}^r(\bb{C})$. Then there exists some $c> 0$ such that for all $|\theta| > 0$, we have
$$|f_{\theta}(x) - f(x)| \leq c |\theta|^{\kappa}$$
\end{lemma}

\begin{proof}
We first claim that for $|y|$ sufficiently small, the automoprhism $\tau_y$ has Lipschitz constant $O(|y|)$ with respect to our metric. Indeed if we let $\eps = |y|$, then our automorphism $\tau_y$ has entries $1+O(\eps)$ on the diagonal and $O(\eps)$ everywhere else. Then if $[z_0:\cdots:z_r]$ is our point, which we may WLOG assume that $|z_r|$ is maximal, then the image under $\tau_y$ is of the form $[z_0 + O(\eps|z_r|): \cdots : z_r + O(\eps|z_r|)]$. Passing into the chart $\{z_r = 1\}$ then gives us a point which is $O(\eps)$ distance away from our original point in our metric. 
\par 
Now by definition, we have
$$f_{\theta}(x) = \int_{\PGL_{n+1}(\bb{C})} f(\tau_y(x)) d \rho_{\theta}(y).$$
The measure $\rho_{\theta}(y)$ is supported on $|y| < \theta$ and for such $y$, we have $|\tau_y(x)-x| < O(\theta)$. Using Holder continuity, we get that $|f(\tau_y(x)) - f(x)| < O( \theta^{\kappa})$ and so we obtain
$$\left|f_{\theta}(x) - f(x) \right| \leq O(\theta^{\kappa})$$ 
as desired.
\end{proof}

We now also show that if $i \partial \ovl{\partial} f$ is not too negative, then so is $i \partial \ovl{\partial} f_{\theta}$. Let $\omega$ be some smooth positive $(1,1)$-form on $\bb{P}^r$. 

\begin{lemma} \label{Regularization3}
Assume that $\del \ovl{\del} f \geq -\omega$. Then for all $\theta > 0$ sufficiently small, there exists some constant $c > 0$, depending on $\omega$, such that 
$$\del \ovl{\del} f_{\theta} \geq -(1 + c \theta) \omega.$$
\end{lemma}

\begin{proof}
We have 
$$\del \ovl{\del} f_{\theta} = \int_{\PGL_{r+1}(\bb{C})} \del \ovl{\del} (f \circ \tau_y) d \rho_{\theta}(y)$$
and so it suffices to show that $\del \ovl{\del} (f \circ \tau_y) \geq -(1 + O(\theta)) \omega$ for any $y$ satisfying $|y| < \theta$. Fix any strongly positive $(r-1,r-1)$-form $\eta$. Let $\tau_{y'}$ denote the inverse of $\tau_y$. We have
$$\langle \del \ovl{\del}(f \circ \tau_y), \eta \rangle = \langle f \circ \tau_y , \ovl{\del} \del \eta \rangle = \langle (\tau_{y'})_* f, \ovl{\del} \del \eta \rangle = \langle f, (\tau_{y'})^* \ovl{\del} \del \eta \rangle = \langle f, \ovl{\del} \del ( (\tau_{y'}^*) \eta) \rangle$$
$$= \langle \del \ovl{\del} f, (\tau_{y'})^* \eta \rangle \geq \langle - \omega, (\tau_{y'})^* \eta \rangle = \langle (\tau_{y'})_* (-\omega), \eta \rangle,$$
where we used the fact that $f \circ \tau_y$ is the pushforward of $f$ under the inverse map $\tau_{y'}$ and that the pullback of a strongly positive form under $\tau_{y'}$ is still strongly positive as $\tau_{y'}$ is a holomorphic map. 
\par 
It now suffices to show that 
$$(\tau_{y'})_* (-\omega) + (1 + O(\theta)) \omega$$
gives a positive definite matrix at each point $x$. Let $\theta = \eps$. Then in a local chart, if 
$$\omega(z) = \sum_{1 \leq i,j \leq n} c_{i,j}(z) dz_i \wedge d\ovl{z}_j,$$
we have 
$$(\tau_{y'})_*(\omega)(z) = \sum_{1 \leq i,j \leq n} c_{i,j}(\tau_y(z)) dz_i \wedge d \ovl{z}_j.$$
Then $|\tau_y(z) - z| = O(\eps)$, and so $|c_{i,j}(z) - c_{i,j}(\tau_y(z))| < O(\eps)$ too. If $A(z)$ is the matrix formed by $(c_{i,j}(z))$, then we know that $|A(\tau_y(z)) -A(z)|$ is a matrix with entries bounded by $O(\eps)$. Since $A(z)$ is positive definite, we can find a bound $\lambda$ for the smallest eigenvalue. Then $A(\tau_y(z))-A(z) \geq - O(\lambda^{-1} \eps)A(z)$ and so locally, we have $(\tau_{y'})_* \omega \geq -(1 + O(\theta)) \omega$. We may then make the constants uniform by taking a finite cover $\{U_i\}$ with each $U_i$ living inside a larger chart $V_i$ where $U_i \simeq \bb{D}_1$ and $V_i \simeq \bb{D}_3$.  
\end{proof}

Let $X = \bb{P}^r$ and $L = O(1)$. We may then give $L$ the Fubini-Study metric $h_{\FS}$ so that $i \partial \ovl{\partial} h_{\FS} = \omega_{\FS}$ is a positive $(1,1)$-form. Now if $h$ is our continuous metric on $L$, then $f = h - h_{\FS}$ is a globally defined continuous function such that $i \partial \ovl{\partial} f = dd^c h - \omega_{\FS}$. We may then apply our regularization process to $f$ to obtain $f_{\theta}$, and set $h_{\theta} = f_{\theta} + h_{\FS}$ to be our new hermitian metric. If $h$ is semipositive, so that $i \partial \ovl{\partial} f \geq - \omega_{\FS}$, by Lemma \ref{Regularization3} we know that 
$$i \partial \ovl{\partial} f_{\theta} \geq -(1 + O(\theta)) \omega_{\FS} \implies i \partial \ovl{\partial} h_{\theta} \geq -O(\theta) \omega_{\FS}$$
and so our regularized metric is still nearly semipositive. 

\subsection{Normed modules and volumes} We will mainly refer to \cite[2.3]{Yua08} and \cite[Chapter 3]{Mor14}. Let $M$ be a free $\bb{Z}$-module. If $| \cdot |$ is a norm on $M \otimes_{\bb{Z}} \bb{R} = M_{\bb{R}}$, we call $(M, | \cdot |)$ a normed $\bb{Z}$-module. Let $B_M = \{ x \in M_{\bb{R}} \mid |x| \leq 1\}$. After fixing a choice of a Haar measure on $M_{\bb{R}}$, we may define the Euler characteristic $\chi(M, | \cdot |)$ by the formula
$$\chi(M, | \cdot |) = \log \vol(B_M) - \log \vol(M)$$
where $\vol(M)$ denotes the volume of a fundamental domain of $M$. This is independent of the choice of our Haar measure. The following proposition follows from Minkowski's theorem.

\begin{lemma} \label{Modules1}
Assume that $\chi(M ,| \cdot |) > \dim M_{\bb{R}} + \log 2$. Then the intersection $B_M \cap M$ is non-empty.
\end{lemma}

\begin{proof}
This is just Minkowski's theorem.    
\end{proof}

Now let $| \cdot |_1$ and $| \cdot |_2$ be two different norms on $M_{\bb{R}}$. We would like to calculate the difference between $\chi(M, | \cdot |_1)$ and $\chi(M, | \cdot |_2)$. When both our norms are given by inner products $\langle \, , \, \rangle_1$ and $\langle \, , \, \rangle_2$, the next proposition gives a way to calculate this difference. 

\begin{proposition} \label{Modules2}
Let $| \cdot |_1$ and $| \cdot |_2$ be two different norms on $M_{\bb{R}}$ that are given by inner products $\langle \, , \, \rangle_1$ and $\langle \, , \, \rangle_2$. Pick an orthonormal basis $e_1,\ldots,e_n$ for $\langle , \rangle_1$ that is orthogonal for $\langle , \rangle_2$. Then
$$\chi(M, | \cdot |_1) - \chi(M, | \cdot |_2) = \sum_{i=1}^{n} \log |e_i|_2.$$
\end{proposition}

\begin{proof}
By definition, it suffices to calculate the difference of volumes between the unit ball for $| \cdot |_2$ and $| \cdot |_1$. Fix a Haar measure such that the volume of $\{\lambda_1 e_1 + \cdots + \lambda_n e_n \mid 0 \leq \lambda_i \leq 1\}$ is $1$. Then the volume of the unit ball for $| \cdot |_1$ is exactly the volume of the $n$-sphere. 
\par 
For $| \cdot |_2$, the unit ball is given by  
$$\{(\lambda_1,\ldots,\lambda_n) \mid  \sum_{i=1}^{n} \lambda_i^2 |e_i|_2^2 \leq 1\}$$
as a subset of $\bb{R}^n$. The volume of this is $\prod_{i=1}^{n} |e_i|^{-1}_2$ multiplied by the volume of the n-sphere. Hence taking $\log$'s and the difference, we obtain 
$$\chi(M, | \cdot |_1) - \chi(M, | \cdot |_2) = \sum_{i=1}^{n} \log |e_i|_2$$
as desired.
\end{proof}

We will be using Proposition \ref{Modules2} repeatedly to compute difference of volumes in later sections. An immediate consequence of Proposition \ref{Modules2} is that if $| \cdot |_2 = \alpha | \cdot |_1$ for some $\alpha > 0$, then 
$$\chi(M, | \cdot |_1) -\chi(M, | \cdot |_2) = \log \alpha \dim M_{\bb{R}}.$$
The situation where our normed modules arise comes from looking at a lattice inside the space of global sections $H^0(X,L)$ for a projective variety $X$ over a number field $K$. For each archimedean place $v$ of $K$, we have an embedding $v: K \xhookrightarrow{} \bb{C}$ which gives an inclusion $H^0(X, L) \xhookrightarrow{} H^0(X(\bb{C}_v),L)$. We may combine them to get an isomorphism
$$H^0(X,L) \otimes_{\bb{Q}} \bb{C} \simeq \bigoplus_{v \text{ arch. }} H^0(X(\bb{C}_v),L)$$
where we count conjugate archimedean places $v$ separately. 
\par 
Now if each $H^0(X(\bb{C}_v),L)$ has an inner product $\langle \, , \, \rangle_v$, we may define an inner product $\langle \, , \, \rangle$ on $H^0(X,L) \otimes_{\bb{Q}} \bb{C}$ by declaring each $H^0(X(\bb{C}_v),L)$ to be orthogonal to each other. If each $\langle \, , \, \rangle_v$ is positive definite, then this inner product $\langle \, , \, \rangle$ is also positive definite by \cite[Proposition 3.9]{Mor14}. Furthermore, if one knew that for conjugate archimedean places $v,v'$, we have 
$$\langle x , y \rangle_v = \overline{\langle x ,y \rangle_{v'}}$$
for $x,y \in H^0(X,L)$, then one can check that we get an induced inner product on $H^0(X,L) \otimes_{\bb{Q}} \bb{R}$. 
\par 
To get the inner products $\langle \, , \, \rangle_v$, we will choose a positive measure $\mu_v$ on $X(\bb{C}_v)$ for each archimedean place $v$, along with a hermitian metric on $L$. Then the formula
$$\langle s , s' \rangle_v = \int_{X(\bb{C}_v)} \langle s,s' \rangle_x d \mu_{v}(x)$$
where $\langle \, , \, \rangle_x$ is the inner product on the fibers coming from our hermitian metric produces an inner product. We will call this a $L^2$-norm due to integrating over a measure. 
\par 
The other situation where a norm appears is the supremum norm. For each archimedean place $v$, we have a metrized line bundle $\ovl{L}$. For a global section $s$, we set $|s|_v = \sup_{x \in X(\bb{C}_v)} |s(x)|_x$ where $| \cdot |_x$ is again the fiberwise norm. This gives a norm on each $H^0(X(\bb{C}_v),L)$ and taking the maximum over each archimedean place $v$ gives a supremum norm on $H^0(X,L) \otimes_{\bb{Q}} \bb{C}$. 

\section{Gromov's Norm Comparison Theorem}
Let $X$ be a compact complex manifold of dimension $r$ and let $\cal{U}_I = \{U_{\alpha}\}_{\alpha \in I}$ be a choice of good charts for $X$. Let $L$ be a metrized line bundle on $X$ and fix a choice of non-vanishing sections $t_{\alpha}$ on each $U_{\alpha}$ for $L$. For the line bundle $L^n$, we will use $t_{\alpha}^n$ to trivialize it. As a reminder, a choice of good charts means that each $U_{\alpha}$ is biholomorphic to the open polydisc of radius $3$ and every point $x \in X$ lies inside the unit polydisc for one of these charts. 
\par 
In this short section, we will prove a quantitative version of Gromov's comparison between the $L^2$-norm and the sup norm of a given section $s \in H^0(X,L^n)$, as $n$ gets large. Given a section $s$, if $s_{\alpha}$ is the holomorphic function representing $s$ on $U_{\alpha}$ under our trivialization, we know that $|s(z)|$ is given by the formula
$$|s(z)|^2 = e^{h_{\alpha}}(z) s_{\alpha}(z) \ovl{s}_{\alpha}(z)$$
for all $x \in U_{\alpha}$.

\begin{proposition} \label{Gromov1}
Let $\mu$ be a positive volume form on $X$ with total volume $1$. Let $L$ be a line bundle and $\ovl{L}$ be a metrization with metric $h$. Let $c$ be a constant such that $d \mu > c dV$ when restricted to $\bb{D}_2$ for each of the charts $z_{\alpha}(U_{\alpha})$, where $dV$ denotes the Lebesgue measure and let $1 \geq c_2 > 0$ be a constant such that $e^{h_{\alpha}(x)} \geq c_2$ for all $x \in U_{\alpha}$ such that $|z_{\alpha}(x)| \leq 2$. Then for all $n \geq 1$, we have the inequality
$$||s||_{\sup} \leq O\left( \left(\frac{c_{e^h,1}}{c_2} \right)^{2r} \frac{n^r}{c} ||s||_{L^2}\right) \text{ for all } s \in H^0(X,L^n)$$
where we take the $L^2$-norm with respect to $\mu$.  
\end{proposition}

\begin{proof}
The proof is mainly the same as Proposition 2.13 of \cite{Yua08}. Note that $h_{\alpha}$ in Yuan's Proposition is our $e^{h_{\alpha}}$. We reproduce a sketch of the argument here. If $h_{\alpha}$ is the function induced by $h$ on $U_{\alpha}$, then for $x \in U_{\alpha}$, one has 
$$|s(z)|^2 = e^{n h_{\alpha}(z)}  s_{\alpha}(z) \ovl{s}_{\alpha}(z) \text{ for all } s \in H^0(X,L^n).$$
Identifying $U_{\alpha}$ with $\bb{D}_3 = \{z \in \bb{C}^r \mid |z| < 3\}$. the constant $c_{e^h,1}$ serves as a Lipschitz constant for $e^{h_{\alpha}}$ in the disc $\bb{D}_2 = \{|z| \leq 2\}$ and so $|e^{h_{\alpha}(z_1)} - e^{h_{\alpha}(z_0)}| \leq c_{e^h,1}|z_1 - z_0|$ for $z_1,z_0 \in \bb{D}_2$. Hence
$$e^{h_{\alpha}(z_1)} \geq e^{h_{\alpha}(z_0)} - c_{e^h,1}|z_1 - z_0| \geq e^{h_{\alpha}(z_0)}\left(1 - \frac{c_{e^h,1}}{c_{2}} |z_1 - z_0| \right).$$
This implies that in Yuan's argument, we may take $c_1 = 2 \max \{1, \frac{c_{e^h,1}}{c_2}\}$. The constant $c''$ in Yuan's argument may also be taken as $c$ multiplied by some constant depending only on the dimension $r$. The rest of Yuan's argument then gives us the inequality
$$||s||^2_{L^2} \geq c'' c_1^{-2r} \left( \frac{||s||^2_{\sup}}{(2r + n)\binom{2r + k-1}{2r-1}} \right)$$
which gives us what we want. 
\end{proof}

This gives us the following comparison of volumes between the $L^2$-norm and the sup norm. Given a positive measure $\mu$, we may use it to define an inner product on the sections of $H^0(X,L^n)$ by $\langle s,s' \rangle = \int_X s(z) \ovl{s'}(z) d \mu$, which we will refer to as the $L^2$-norm. 

\begin{corollary} \label{Gromov2}
Keep the same notation as in Proposition \ref{Gromov1}. If $V_{n,L^2}$ denotes the unit ball in $H^0(X,L^n)$ for the $L^2$-norm, and $V_{n,\sup}$ denotes the unit ball under the sup norm, we have the inequality
$$0 \leq \log \vol(V_{n,L^2}) - \log \vol(V_{n,\sup}) \leq O \left( \log \frac{c_{e^h,1} n}{c_2 c'}  \right) \dim H^0(X,L^n).$$
\end{corollary}

\begin{proof}
We certainly have $V_{n,\sup} \subseteq V_{n,L^2}$ because $\mu$ has volume $1$. If we let $c$ denote a constant such that $||s||_{\sup} \leq c||s||_{L^2}$, then we have 
$$V_{n,L^2} \subseteq c V_{n,\sup} \implies \log V_{n,L^2} - \log V_{n,\sup} \leq (\log c )\dim H^0(X,L^n).$$
Applying Proposition \ref{Gromov1} to obtain a constant $c$, the inequality then follows. 
\end{proof}

\section{Quantitative Upper Bounds on Bergman Kernels} \label{sec: QuantBergman}
Again, let $X$ be a compact complex manifold of dimension $r$ and let $\cal{U}_I = \{U_{\alpha}\}_{\alpha \in I}$ be a choice of good charts for $X$. Let $L$ be a metrized line bundle on $X$ and fix a choice of non-vanishing sections $t_{\alpha}$ on each $U_{\alpha}$ for $L$. For the line bundle $L^n$, we will use $t_{\alpha}^n$ to trivialize it. As a reminder, a choice of good charts means that each $U_{\alpha}$ is biholomorphic to the open polydisc of radius $3$ and every point $x \in X$ lies inside the unit polydisc for one of these charts. 
\par 
Given a positive line bundle $\ovl{L}$ with hermitian metric $e^h$, we let $\omega = i \partial \ovl{\partial} h$ be the associated positive $(1,1)$-form and $\mu = \omega^{\wedge r}/r!$ be its volume form. We can then endow $H^0(X,L^n)$ with a $L^2$-norm by 
\begin{equation} \label{eq: InnerProduct1} 
\langle s, s' \rangle = \int_{X(\bb{C})} \langle s,s'\rangle_x d \mu(x)
\end{equation}
where $\langle \, , \, \rangle_x$ is the fiberwise inner product. It is then known, dating back to Tian \cite{Tian90}, that if $s_1,\ldots,s_N$ is an orthonormal basis of sections for $H^0(X,L^n)$, we have
\begin{equation} \label{eq: Bergman1} 
\sum_{i=1}^{N} |s_i(x)|^2 \sim \left(\frac{n}{2 \pi} \right)^{r}
\end{equation}
uniformly for $x \in X(\bb{C})$ as $n \to \infty$. The function $\sum_{i=1}^{N} |s_i(x)|^2$ is independent of the choice of orthonormal basis and is called the Bergman kernel. In this section, we will prove a quantitative version of \eqref{eq: Bergman1} along with a quantitative version of Gromov's norm comparison inequality. Using the approach of Yuan \cite{Yua08} will then allow us to deduce a lower bound on the difference in volumes of the sup unit ball of $H^0(X,L^n)$ when the metric of $L$ is perturbed.
\par 
We will follow \cite{Cha15} to obtain our quantitative version of \eqref{eq: Bergman1}, where the ideas date back to Berndtsson \cite{Ber03}. Dinh--Ma--Nguyen \cite{DMN17} also prove a quantitative version of the asymptotic expansion of Bergman kernels, but is of a different flavor from what we need. Throughout the section, we will assume that our complex manifold $X$ and the line bundle $L$ is fixed. We will be interested in how our constants depend on the metric of $L$, especially on its positivity. 
\par 
Given $x \in X(\bb{C})$ a metrized line bundle $\ovl{L}$ with associated $(1,1)$-form $\omega$, we want to pass to normal coordinates around $x$ for our Kahler form $\omega$. We first control the error in this normal coordinates. We assume that $\omega$ is positive, so that there exists a constant $1 \geq \lambda_{\omega} > 0$ such that in each of the charts $U_{\alpha} \simeq \bb{D}_3$, the eigenvalues of $\omega(z)$ is $\geq \lambda_{\omega}$ with respect to $\sum_{i=1}^{r} dz^i \wedge d \ovl{z}^i$ for $z \in \bb{D}_2$. 

\begin{proposition} \label{NormalCoords1}
For each $x \in X(\bb{C})$, there exists a chart $\vphi_x: V_x \to \bb{C}^r$ such that $x$ maps to the origin, the image $\vphi_x(V_x)$ contains a disc of radius $O(c_{\omega,0})^{-1}$ such that our $(1,1)$-form $\omega$ is of the form
$$\frac{i}{2} \sum_{1 \leq i,j \leq r} g_{i,j}(z) dz^i \wedge d \ovl{z}^j$$
where $g_{i,j}(0) = \delta_{i,j}$ and the order $k$ derivatives for $g_{i,j}(z)$ are bounded by $O(\lambda_{\omega}^{-1}) c_{\omega,k}$. 
\end{proposition}

\begin{proof}
Given $x$, we can find a $U_{\alpha} \simeq \bb{D}_3$ such that $x$ is contained in $\bb{D}_1$. Thus there is an unit disc around $x$ that is contained in $\bb{D}_2$. Re-centering $x$ to be $0$, We write
$$\omega = \frac{i}{2} \sum_{1 \leq i,j \leq r} f_{i,j}(z) dz^i \wedge d \ovl{z}^j.$$
Then we know that $\{f_{i,j}(0)\}_{i,j}$ forms a positive definite hermitian matrix. Thus if we let $A$ denote our hermitian matrix, there exists some unitary $P$ such that $PA \ovl{P}^T = D$ where $D$ is the diagonal matrix of eigenvalues for $\omega$ at $0$. We then scale the $i$th column of $P$ by $\sqrt{\lambda_i^{-1}}$ where $\lambda_i$ is the $i$th eigenvalue to get an orthogonal $Q$ such that $QA \ovl{Q}^T = I$. Then the entries of $Q$ are upper bounded by $O(\lambda_{\omega}^{-1/2})$ and so in our new basis, we get 
$$\omega = \frac{i}{2} \sum_{1 \leq i,j \leq r} g_{i,j}(z) dz^i \wedge d\ovl{z}^j$$
with $g_{i,j}(0) = \delta_{i,j}$ and where the derivatives of $g_{i,j}$ are bounded by that of $f_{i,j}$ multiplied by $r \lambda_{\omega}^{-1/2}$. Since $\ovl{P}^T$ is unitary, it maps $\bb{D}_1$ into $\bb{D}_{r}$ and so the image of its inverse, $P$, when applied to $\bb{D}_1$ will contain $\frac{1}{r} \bb{D}_{1}$. Then the image of $\ovl{Q}^T(\bb{D}_1)$ contains $\frac{1}{r \sqrt{\lambda_1}} \bb{D}_1$. We may bound $\lambda_1$ by the trace of our matrix as our matrix is positive definite, which is $O(c_{\omega,0})$ as desired.  
\end{proof}

This allows us to prove an upper bound on the Bergman kernel, uniform over $x \in X$, where the dependence on how large $n$ has to be on the metric is explicit.

\begin{proposition} \label{Bergman2}
Let $\ovl{L}$ be a positive line bundle with $(1,1)$-form $\omega = i \partial \ovl{\partial} h$ and let $s_1,\ldots,s_N$ be an orthonormal basis for $L^n$ under the volume form $\mu = (\omega)^{\wedge r}/r!$. Then for all $n \geq O(c_{\omega,1} \lambda_{\omega}^{-1} c_{h,3})^{8}$, we have
$$\sum_{i=1}^{N} |s(x)|^2 \leq \left(\frac{n}{2 \pi} \right)^r \left(1  + \frac{1}{n^{1/8}}\right) $$
for all $x \in X$. 
\end{proposition}

\begin{proof}
First, it is well known, see \cite[Lemma 3.1]{Cha15} that 
$$\sum_{i=1}^{N} |s_i(x)|^2 = \sup_{s \in H^0(X,L^n)} \frac{|s(x)|}{|s|_{L^2}}$$
where the supremum is taken over all non-zero sections $s$. By Proposition \ref{NormalCoords1}, we can find a local chart of radius $O(c_{\omega,0})^{-1}$ for which 
$$\omega(z) = \frac{i}{2}\sum_{1 \leq i,j \leq r} \left( f_{i,j}(z)\right) dz^i \wedge d\ovl{z}^j$$
where $f_{i,j}(0) = \delta_{i,j}$ and we have a bound on the derivatives of $f_{i,j}(z)$. Now given a local non-vanishing section $s$ in the chart, we get that $-2i dd^c \log |s| = \omega$. Representing $s$ by a holomorphic function $f$, we get $-2i dd^c (\log |f| + h) = \omega$ where $h$ is our hermitian metric. We may multiply $f$ by any local non-vanishing holomorphic function and thus we may add or subtract any harmonic function from $h$. Since we have 
$$dd^c (2 h + \sum_{i=1}^{r} |z_i|^2) = \frac{i}{2} \sum_{1 \leq i,j \leq r} (-f_{i,j}(z) + \delta_{i,j}) dz^i \wedge d \ovl{z}^j,$$
with $f_{i,j}(z) - \delta_{i,j}$ vanishing at $0$, we may modify $h$ by a polynomial of degree $\leq 2$ and coefficients $\leq c_{\omega,2} O(\lambda_{\omega}^{-r})$ in size such that $h + \frac{1}{2} \sum_{i=1}^{r} |z_i|^2$ vanishes at $0$ to order $2$. To bound the third derivative of $h$ in this chart, we observe that the coordinates was transformed by a linear transformation with coefficients $\leq O(\lambda_{\omega}^{-1})$. This changes the derivative of $h$ by at most a factor of $O(\lambda_{\omega}^{-1})$ and so the third derivative of $h - \sum_{i=1}^{r} |z_i|^2$ may be bounded by $O(\lambda_{\omega}^{-1}) c_{h,3}$. Thus we get
$$h = -\frac{1}{2}\sum_{i=1}^{r} |z_i|^2 + O(\lambda_{\omega}^{-1} c_{h,3}) |z|^3.$$
Now let $\vphi(z) = \frac{1}{2} \sum_{i=1}^{r} |z_i|^2$ and let $c_n$ be some constant depending on $n$. Given any section $s$ represented by a holomorphic function $f$ in our local chart, we have 
$$|s(x)|^2 = |f(0)|^2 \leq \frac{ \int_{|z| < \frac{c_n}{\sqrt{n}}} |f(z)|^2 e^{-n \vphi(z)} dV}{\int_{|z| < \frac{c_n}{\sqrt{n}}} e^{-n \vphi(z)} dV}$$
where $dV$ is the usual Lebesgue measure on $\bb{C}^r$. The inequality follows from the fact that $|f(z)|^2$ is plurisubharmonic and that $\vphi(z)$ is constant when each $|z_i|$ is constant. Now if $\mu$ denotes the volume form induced by $\omega$ in our local chart, we see that if $dV$ denotes the Lebesgue measure, then locally 
\begin{equation} \label{eq: BergmanVolume1}
\mu(z) = (1 + O(\lambda_{\omega}^{-1} c_{\omega,1}) |z|)) dV,
\end{equation}
as the the first derivative of the coefficients of $\omega$ in our chart is bounded by $O(\lambda_{\omega}^{-1} c_{\omega,1})$. 
We also have
$$e^{-n \vphi(z)} \leq e^{nh} \cdot e^{n O(\lambda_{\omega}^{-1} c_{h,3})|z|^3}$$
Taking $c_n = n^{1/10}$, we get $c_n/\sqrt{n} = n^{-2/5}$ and so for $|z| \leq \frac{c_n}{\sqrt{n}}$, we have $n|z|^3 \leq n^{-1/5}$ which ensures that $e^{nO(\lambda_{\omega}^{-1} c_{h,3} |z|^3} = O(\frac{1}{n})$. Hence
\begin{equation} \label{eq: Bergman2} \int_{|z| < \frac{c_n}{\sqrt{n}}} |f(z)|^2 e^{-n \vphi(z)} dV \leq  (1+ O(\lambda_{\omega}^{-1} c_{\omega,1})n^{-2/5}) |s|_{L^2}
\end{equation}
since $n \geq O(\lambda_{\omega}^{-1} c_{h,3})^8$. On the other hand, we have
$$\int_{|z| < \frac{c_n}{\sqrt{n}}} e^{-n \vphi(z)} dV = \frac{1}{n^r} \int_{|z| < n^{1/10}} e^{-\frac{1}{2}\sum_{i=1}^{r} |z_i|^2} = \frac{1}{n^r} \left( \int_{|z| \in \bb{C}^r} e^{-\frac{1}{2} \sum_{i=1}^{r} |z_i|^2} - \int_{|z| > n^{1/10}} e^{-\frac{1}{2}\sum_{i=1}^{r} |z_i|^2} \right)$$
$$= \frac{1}{n^r} \left( \left( 2 \pi \right)^r + O(\frac{1}{n}) \right)$$
Hence we obtain
$$|s(x)|^2 \leq \left(\frac{n}{2 \pi} \right)^r \left(1 + O(\lambda_{\omega}^{-1} c_{\omega,1})n^{-2/5} \right) |s|^2_{L^2} \leq \left(\frac{n}{2 \pi} \right)^r \left(1 + \frac{1}{n^{1/8}} \right) |s|^2_{L^2}$$
since $n \geq O(\lambda_{\omega}^{-1}c_{\omega,1})^8$, as desired.
\end{proof}

\section{Global Bound on Differences of Volumes}
Let $X$ be a projective variety over a number field $K$ and let $\vphi: X \to X$ be a polarizable endomorphism. By \cite{Fak14}, after replacing $L$ with a high enough power, we may assume that $L$ embeds $X$ into $\bb{P}^N$ with $L$ being the restriction of $O(1)$, and that $\vphi$ extends to a morphism $\tilde{\vphi}: \bb{P}^N \to \bb{P}^N$ of degree $d$. We may then lift $\tilde{\vphi}$ to a homogeneous map $F: \bb{A}^{N+1} \to \bb{A}^{N+1}$. 
\par 
For each place $v$ of $M_K$, we may define a homogeneous Green's function using $F$ by
$$G_{v}(x) = \lim_{n \to \infty} \frac{\log ||F^n(x)||_v}{d^n} \text{ for } x \in \bb{A}^{N+1}(\bb{C}_v),$$
where $||(x_0,\ldots,x_{n+1})||_v = \max\{|x_0|_v,\ldots,|x_{n+1}|_v\}$. We can then use $G_v$ to define a metric on $L$ for each place $v$ by 
$$\log |s(x)|_v = \log |P(\tilde{x})|_v - G_v(\tilde{x}) \text{ for } x \in \bb{P}^N(\bb{C}_v),$$
where $\tilde{x}$ is any lift of $x$ to $\bb{A}^{N+1}(\bb{C}_v)$. Due to the homogeneity of $G_v$, our definition is independent of the lift. Restricting to $X$, this turns $L$ into a semipositive adelic line bundle $\ovl{L}$ in the sense of Zhang \cite[Section 3.5]{Yua08}, although it would not be important for us to know this. It is easy to see that its associated height on $X(\ovl{K})$ is the canonical height $\h_{\vphi}(x)$, which may also be defined as 
$$\h_{\vphi}(x) = \lim_{n \to \infty} \frac{h(\vphi^n(x))}{d^n}.$$
where $h(x)$ is the standard Weil height on $\bb{P}^N$. Let $L_n = L^{d^n}$. Then there is an induced adelic meric on $L_n$ coming from $L$ and we will denote the adelic line bundle by $\ovl{L}_n$ and the metrized line bundle on $X(\bb{C}_v)$ by $\ovl{L}_{n,v}$. 
\par 
The global sections of $L_n$, denoted by $H^0(X,L_n)$, form a $K$-vector space. We let $M_n$ be the sub $O_K$-module of sections $s$ satisfying $|s|_{v.\sup} \leq 1$ for all non-archimedean places $v$, where $\sup$ denotes the supremum over $X(\bb{C}_v)$. Then $M_n$ is a $\bb{Z}$-submodule of maximal rank of $H^0(X,L_n)$. 
\par 
We now consider $V_n = H^0(X,L_n) \otimes_{\bb{Q}} \bb{R}$. Then $V_n$ is a $\bb{R}$-vector space and we can give it a sup norm as follows. We have 
$$V_n \otimes_{\bb{R}} \bb{C} \simeq \bigoplus_{v \text { arch. }} H^0(X(\bb{C}_v),L_{n})$$
where $X(\bb{C}_v)$ denotes the complex variety obtained by embedding $K$ into $\bb{C}$ via the place $v$. Each $H^0(X(\bb{C}_v),L_{n})$ has a supremum norm coming from the metric on $L_n$ for the place $v$, and taking the supremum of each of them defines a sup norm on $V_n \otimes_{\bb{R}} \bb{C}$ and hence on $V_n$ which we will denote by $| \cdot |_{\sup}$. 
\par 
We now let $C_{n,\sup}$ be the unit ball of sections $s \in V_n$ satisfying $|s|_{\sup} \leq 1$. Fix a Haar measure on $V_n$ and let $\vol M_n$ denote the volume of the fundamental domain of the lattice $M_n$ in $V_n$. The quantity
$$\log \vol C_{n,\sup} - \log \vol M_n$$
does not depend on the choice of our Haar measure. Using arithmetic Hilbert--Samuel \cite[Theorem 1.4]{Zha95}, one can deduce that 
\begin{equation} \label{eq: VolumeDifference1}
\lim_{n \to \infty} \frac{1}{d^n \dim V_n} \left(\log \vol C_{n,\sup} - \log \vol M_n \right) = 0.
\end{equation}

Our main result in the section is a quantitative lower bound on the volume difference \eqref{eq: VolumeDifference1}. The main idea is a generalization of a construction of Baker \cite{Bak06} to higher dimensions due to Looper \cite{Loo24}. 

\begin{theorem} \label{QuantProposition1} 
Let $K = \bb{Q}(\alpha)$ for some $\alpha \in K$. Assume that $n$ is sufficiently large depending on $L$. Then there exists $c = c(\vphi) > 0$ such that for all $n > 0$, we have 
$$\frac{1}{\dim V_n} \left(\log \vol(C_{n,\sup}) - \log \vol(M_n) \right) \geq -c \left(d^{n/2} + [K:\bb{Q}]^2 h(\alpha) \right).$$

\end{theorem}

\begin{proof}
To prove this, we follow an idea of Looper \cite{Loo24} by constructing a special basis. In the case of $\bb{P}^1$, this has been done by Baker \cite{Bak06}. We let $F_i^{(m)}$ be the $i^{th}$ coordinate of the iterate $F^{(m)}$. Since $\Res(F^{(m)})$ is non-zero, it follows by the fundamental theorem of elimination theory \cite{Mac03} that every homogeneous polynomial of degree $ N \geq (r+1)d^{m}$ can be written as
$$\eta_0 F_0^{(m)} + \cdots + \eta_r F_r^{(m)}$$
where $\eta_i$ is a polynomial of degree $\leq  N-d^m$. We can now repeat this with each $\eta_i$, and in particular if we let $S$ be the set of all possible products of $F_i^{(m)}$, we obtain that every homogeneous polynomial of degree $N \geq (r+1)d^m$ can be written as a sum of terms $\eta_i G_i$ for $G_i \in S$ where each $\eta_i$ is a homogeneous polynomial of degree $\leq (r+1)d^m$.
\par 
Hence we can find a basis $\cal{B}$ over $K$ for homogeneous polynomials of degree $d^n$ with coefficients in $K$ where each term is of the form $\eta G$, where $\eta$ is a monomial of degree $\leq (r+1)d^m$ and $G$ is a product of terms of the form $F_i^{(m)}$. To find a basis over $\bb{Q}$, we allow $\eta$ to have a coefficient from $\{1,\alpha,\ldots,\alpha^{l-1}\}$ where $l = [K:\bb{Q}]$. We first bound the sup norm of each term $\eta G$ under $| \cdot |_v$ for each place $v$. 

\begin{proposition} \label{QuantProposition2}
Assume that $m \leq \frac{n}{2}$. For each place $v$ of $K$, if $s$ is the section corresponding to $\eta G$, then we have 
$$\log |s|_{\sup,v} \leq d^{n-m} O_{\vphi}(1) + [K:\bb{Q}] \log^+|\alpha|_v$$
where we view $\eta G$ as a section of $H^0(X,L_n)$. 
\end{proposition}

\begin{proof}
Recall that we have the formula
$$\log |s(x)|_v = \log |\eta(\tilde{x})|_v +  \log|G(\tilde{x})|_v - d^n G_v(\tilde{x})$$
where $\tilde{x} \in \bb{A}^{N+1}(\bb{C}_v)$ is any lift of $x \in \bb{P}^N(\bb{C}_v)$. By \cite[Theorem 5]{Ing22}, we know that 
$$\log ||F(y)|| - d \log ||y|| = O_{F}(1)$$
for any $y \in \bb{A}^{N+1}(\bb{C}_v)$. Applying Tate's telescoping argument tells us that if 
$$G_{v,m} = \frac{\log ||F^m(x)||}{d^m},$$
then
$$|G_{v,m}(x) - G_{v}(x) | \leq \frac{1}{d^m}O_{\vphi}(1).$$
Writing $G$ as a product of $F_i^{(m)}$, as we have $|F_i^{(m)}|_v \leq ||F^{(m)}||_v$, we get
$$\log |\eta(\tilde{x})|_v + \log |G(\tilde{x})|_v - d^n G_v(\tilde{x}) \leq \log |\eta(\tilde{x})|_v - (d^n - \deg G) G_v(\tilde{x}) + O_{\vphi}(d^{n-m}).$$
Pick a lift $\tilde{x}$ such that $||\tilde{x}|| = 1$. Then we may bound
$$\log |\eta(\tilde{x})|_v \leq [K:\bb{Q}] \log^+|\alpha|_v$$
and by \cite[Lemma 6]{Ing22} we have
$$G_v(\tilde{x}) \geq -O_F(1).$$
Since $d^n - \deg G \leq (r+1)d^m$, we have
$$\log|\eta(\tilde{x})| + \log |G(\tilde{x})|_v - d^n G_v(\tilde{x}) \leq [K:\bb{Q}] \log^+|\alpha|_v  + O_{\vphi}(d^{n-m} + d^m)$$
$$= d^{n-m} O_{\vphi}(1) + [K:\bb{Q}] \log^+|\alpha|_v)$$
as desired.

\end{proof}

Now let $M'_n$ be the lattice spanned by our basis $\cal{B}$. We will use it to approximate our lattice $M_n$. For each non-archimedean place $v$, if $p$ is the corresponding prime, it follows from Proposition \ref{QuantProposition2} that for any section $s \in M_n'$, we have 

$$\left(p^{O_{\vphi}(d^{n-m}) +  [K:\bb{Q}] \log^+|\alpha|_v} \right) s \subseteq M_{n,v}$$
where $M_{n,v}$ is the $\bb{Z}_p$-submodule of sections satisfying $|s|_{v,\sup} \leq 1$ for that particular non-archimedean place $v$. When $p$ is a prime of good reduction for $F$ and also satisfying $|\alpha|_v = 1$, we have $\log |s|_{\sup,v} \leq 1$. Hence summing up over all places of bad reduction and those for which $|\alpha|_v \not = 1$, we obtain that 
$$\log \vol(M_n) \leq (\dim V_n) \left(O_{\vphi}(d^{n-m}) + [K:\bb{Q}]^2 h(\alpha) \right)  + \log \vol(M'_n).$$
We now look at the fundamental domain $\cal{F}$ of $M'_n$ that is given by our basis $\cal{B}$. Any element of our fundamental domain can be written as 
$\sum_i c_i \eta_i G_i$
where $0 \leq c_i \leq 1$. Applying our bound and using the triangle inequality, we get 
$$ \log |\sum_i c_i \eta_i G_i|_v \leq \left( O_{\vphi}(d^{n-m}) + [K:\bb{Q}]\log^+|\alpha|_v \right) + \log \dim V_n$$ 
$$= O_{\vphi}(d^{n-m}) + [K:\bb{Q}] \log ^+|\alpha|_v  + \log \dim V_n.$$
Hence if we scale $C_n$ by $e^{O_{\vphi}(d^{n-m}) + [K:\bb{Q}]^2 h(\alpha) + \log \dim V_n}$, it would contain our fundamental domain $\cal{F}$.  
We thus obtain that 
$$\log \vol (C_n) - \log \vol(M'_n) \geq -O_{\vphi}(d^{n-m} + [K:\bb{Q}]^2 h(\alpha)) \dim(V_n) - (\log \dim V_n) \dim V_n.$$
Putting things together, we obtain
$$\frac{1}{\dim V_n} \left(\log \vol(C_n) - \log \vol (M_n) \right) \geq -O_{\vphi}(d^{n-2k} + [K:\bb{Q}]^2 h(\alpha)) - \log \dim V_n$$
which is an error of $-\left(O_{\vphi}(d^{n-m}) + [K:\bb{Q}]^2 h(\alpha) \right)$ since, $\log \dim V_n \leq d^{n-m}$ for $n$ large enough depending on $L$, as desired. We then obtain an exponent of $\frac{n}{2}$ by taking $m = \frac{n}{2}$. 
\end{proof}

\section{Quantitative Equidistribution for Abelian Varieties} \label{sec: QuantEquibAbelian}
Our aim in this sectin is to prove the main theorem of our paper in the case where $X$ is an abelian variety, so that $\vphi = [2]$ is the doubling map and $\h_{\vphi} = \h_{X}$ is the Neron--Tate height. 

\begin{theorem} \label{AbelianQuantTheorem1}
Let $X$ be an abelian variety over a number field $K$ and $L$ a symmetric ample line bundle on $X$. Let $\h_X(x)$ be the Neron--Tate height associated to $L$. Then there exists constants $c_1 = c_1(K,L), c_3=  c_3(K,L)$ such that for any $\eps > 0$ and embedding $v: K \xhookrightarrow{} \bb{C}$, if $f: X(\bb{C}_v) \to \bb{R}$ is a smooth function, then
\begin{enumerate}
\item there exists a hypersurface $H(f,\eps)$ defined over $K$ which has degree at most $ c_1 \eps^{-8}$ with respect to $L$,
\item for all $x \in X(\ovl{K})$ with $\Gal(\ovl{K}/K)$-orbit $F_x$, if $x \not \in H(f,\eps)$ then

$$\left|\frac{1}{|F_x|} \sum_{y \in F_x} f(y) - \int f(y) d \mu_{X,v}(y) \right| \leq c_{f,3} \left(\frac{\h_{X}(x)}{\eps} + c_3 \eps \right),$$
where $c_{f,3}$ is a constant depending on the derivatives of $f$ up to order $3$ and $\mu_{X,v}$ is the Haar measure for $X$ associated to $v$.  
\end{enumerate}
\end{theorem}

Given our smooth function $f$, we first divide it by $c_{f,3}$ to get a new function $f$ which satisfies $c_{f,3} \leq 2$. It then suffices to prove that 
$$\left|\frac{1}{|F_x|} \sum_{y \in F_x} f(y) - \int f(y) d\mu_{\vphi,v}(y) \right| \leq \frac{\h_{\vphi}(x)}{\eps} + c_3 \eps$$
for functions $f$ satisfying $c_{f,3} \leq 2$ and obtain the general version by multiplying $c_{f,3}$ back in. 
\par 
The strategy of the proof will be the same as the general case of polarized endomorphisms, but there are less technical details due to the smoothness of the metric. 
\par 
Let $L$ be a symmetric ample line bundle on $X$. Then $\vphi$ is polarizable with respect to $L$ and we may find an embedding $X \xrightarrow{} \bb{P}^N$ such that $\vphi$ extends to $\tilde{\vphi}$. 
\par 
We may lift $\tilde{\vphi}$ to a homogeneous map $F: \bb{A}^{N+1} \to \bb{A}^{N+1}$. Then for each place $v \in M_K$, we may define 
$$G_v(x) = \lim_{n \to \infty} \frac{\log ||F^n(x)||}{d^n} \text{ for } x \in \bb{A}^{N+1}(\bb{C}_v).$$
where $||(x_0,\ldots,x_{n+1})||_v = \max\{|x_0|_v,\ldots,|x_{n+1}|_v\}$. We can then use $G_v$ to define a metric on $L$ for each place $v$ by 
$$\log |s(x)|_v = \log |P(\tilde{x})|_v - G_v(\tilde{x}) \text{ for } x \in \bb{P}^N(\bb{C}_v),$$
where $\tilde{x}$ is any lift of $x$ to $\bb{A}^{N+1}(\bb{C}_v)$. This turns $\ovl{L}$ on $\bb{P}^N$ into a semi-positive adelic line bundle in the sense of Zhang. The main advantange is this setting is that when $v$ is an archimedean place, the metric on $\ovl{L}_v$ when restricted to our abelian variety $X \subseteq \bb{P}^N$ is smooth, and the associated $(1,1)$-form is positive. 

\begin{proposition} \label{AbelianMetric1}
The metric on $\ovl{L}_v$ when restricted to $X(\bb{C}_v)$ is smooth. Furthermore, the associated $(1,1)$-form $\omega = c_1(\ovl{L}_v)$ on $X(\bb{C}_v)$ is a positive $(1,1)$-form. 
\end{proposition}

\begin{proof}
For our archimedean place $v$, we first modify $|| \cdot ||_v$ so that the metric is the Fubini--Study metric, i.e. 
$$||F(x_0,\ldots,x_n)||_v = \frac{|F(x_0,\ldots,x_n)|_v}{\sqrt{|x_0|^2 + \cdots + |x_n|^2}}.$$
This does not change the function $G_v(x)$. Then if $\omega$ denotes the Fubini--Study form on $\bb{P}^N$, we know that if $\omega_F = \lim_{n \to \infty} \frac{1}{d^n} (F^n)^* \omega$ where $\omega_F$ is a $(1,1)$-current, then we have 
$$dd^c \log |s(x)|_v = \delta_{\divv(s)} - \omega_F.$$
Now when restricted to $X(\bb{C}_v)$, $\omega_F$ is given by $\lim_{n \to \infty} \frac{1}{d^n} ([2^n])^* \omega$. If we write $X(\bb{C}_v) \simeq \bb{C}^g/\Lambda$ for some lattice $\Lambda$, we may write $\omega$ as 
$$\omega = \sum_{1 \leq i,j \leq n} f_{i,j}(z) dz_i \wedge d \ovl{z}_j.$$
Let $\mu_v$ be the Haar measure on $X(\bb{C}_v)$, normalized to have total volume $1$. It is then easy to verify that the limit of $\omega_F = \lim_{n \to \infty} \frac{1}{4^n} ([2^n])^* \omega$ is given by
$$\omega_F = \sum_{1 \leq i ,j \leq n} c_{i,j} dz_i \wedge d \ovl{z}_j$$
where
$$c_{i,j} = \int_{X(\bb{C}_v)} f_{i,j}(z) d \mu_v$$
Then clearly $\omega_F$ is a smooth positive $(1,1)$-form, which implies that the metric on our line bundle is smooth as desired. 
\end{proof}

Since the metric on $\ovl{L}_v$ is smooth, we do not need to regularize our metric in order to apply the results of Section \ref{sec: QuantBergman}. We then set $L_n = L^{d^n}$ and let $V_n = H^0(X,L_n) \otimes_{\bb{Q}} \bb{R}$. We then have the isomorphism 
$$V_n \otimes_{\bb{R}} \bb{C} \simeq \bigoplus_{v \text { arch. }} H^0(X(\bb{C}_v),L_{n})$$
We may then define a sup norm $| \cdot |_{\sup}$ on $V_n \otimes_{\bb{R}} \bb{C}$ and we let $C_{n,\sup}$ be the unit ball of sections $s$ satisfying $|s|_{\sup} \leq 1$. Similarly, we may let $M_n$ be the $\bb{Z}$-submodule of sections $s$ of $H^0(X,L_n)$ that satisfies $|s|_{\sup,v} \leq 1$ for all non-archimedean places $v$. 
\par 
By Theorem \ref{QuantProposition1}, we have 
\begin{equation} \label{eq: AbelianBound2}
\frac{1}{\dim V_n}  \left( \log \vol (C_{n,\sup}) - \log \vol(M_n) \right) \geq -O_{\vphi,K}(d^{n/2}).
\end{equation}
Now given a measure $\mu_v$ for archimedean place $v$ such that conjugate places share the same measure, we may then define a $L^2$-norm on each $H^0(X(\bb{C}_v),L_n)$ by
$$\langle s,s' \rangle_v = \int_{X(\bb{C}_v)} \langle s,s' \rangle_v d \mu_v$$
and thus on $V_n \otimes_{\bb{R}} \bb{C}$ by declaring each component to be orthogonal. Following \cite[Section 3.3]{Mor14}, our inner product $\langle \, , \, \rangle$ on $H^0(X,L_n) \otimes_{\bb{Q}} \bb{C}$ is of real type and hence it induces one on $V_n$.   
\par 
Fix an archimedean place $v$ and let $f: X(\bb{C}_v) \to \bb{R}$ be our test smooth function, normalized so that $c_{f,3} \leq 2$. We now perturb the metric of $\ovl{L}_n$ at the place $v$ by mutiplying it by $e^{-\eps f}$. If $v$ has a complex conjugate pair $v'$, we also perturb the metric at $v'$ by $e^{-\eps f}$. We denote our new adelic line bundle by $\ovl{L}_n(\eps f)$
\par 
If $\omega$ is the $(1,1)$-form for $\ovl{L}_v$, then as it is a positive $(1,1)$-form, we know that $\lambda_{\omega} > \delta$ for some fixed $\delta > 0$. We take $\eps$ small enough so that if $\omega'$ is the $(1,1)$-form for $\ovl{L}(\eps f)$, then $\lambda_{\omega'} \geq \frac{1}{2} \delta$. In particular, $\mu_v = (\omega'_v)^{\wedge r}$ is a positive volume form. We normalize $\mu_v$ so that $\mu_v$ has total volume $1$. For each other archimedean place $w$, we let $\mu_w$ be some fixed positive volume form with total volume $1$. With this choices of measures $\mu_v$, we can put a $L^2$ norm on $V_n$. 
\par 
We now let $B_{n,\sup}$ be the unit ball of sections for $V_n$ under the sup norm for $\ovl{L}_n(\eps f)$, and $B_{n,L^2}, C_{n,L^2}$ be the unit ball under the $L^2$ norm for $\ovl{L}_n(\eps f)$ and $\ovl{L}_n$ respectively. As each $\mu_w$ has total volume $1$, we have $|s|_{\sup,v} \geq |s|_{L^2,v}$ and so 
$$| \cdot |_{L^2} \leq [K:\bb{Q}]| \cdot |_{\sup}$$
on $V_n$. Thus we trivially have 
\begin{equation} \label{eq: AbelianIneq1}
\frac{1}{\dim V_n} \left(\log \vol C_{n,L^2} - \log \vol C_{n,\sup} \right) \geq -\log [K:\bb{Q}].
\end{equation}
We now have the following lemma.

\begin{lemma} \label{AbelianQuantDifference1} 
Let $n_v = 2$ if $v$ is complex and $1$ if $v$ is real. Then 
$$\log \vol B_{n,L^2} - \log \vol C_{n,L^2} = n_v \left(\log \vol B_{n,v,L^2} - \log \vol C_{n,v,L^2} \right)$$
where $B_{n,v,L^2}$ is the unit ball of sections with $L^2$-norm $\leq 1$ for $H^0(X(\bb{C}_v),L_n)$ with the metric from $\ovl{L}_n(\eps f)$ and similarly for $C_{n,v,L^2}$. 
\end{lemma}

\begin{proof}
Let $\langle \, , \, \rangle$ be the inner product induced by the metric of $\ovl{L}_n$ and $\langle \, , \, \rangle'$ be induced by $\ovl{L}_n(\eps f)$. To calculate $\log \vol B'_{n,L^2} - \log \vol C'_{n,L^2}$, we pick a basis of sections $s_1,\ldots,s_N$ of $V_n$ and we compute the difference 
$$\log \det ( \langle s_i, s_j \rangle') - \log \det ( \langle s_i, s_j \rangle).$$
Tensoring by $\bb{C}$, we may consider a basis inside $V_n \otimes_{\bb{R}} \bb{C} = \bigoplus H^0(X(\bb{C}_w),L_n)$. We may then choose the $s_i$'s such that they each lie entirely in some $H^0(X(\bb{C}_w),L_n)$. Then since our inner products agree if $w$ is not $v$ or its conjugate, our expression reduces to
$$n_v (\log \det( \langle s_i, s_j \rangle'_v) - \log \det (\langle s_i, s_j \rangle_v)) = n_v \left(\log \vol B'_{n,v,L^2} - \log \vol C'_{n,v,L^2} \right)$$
as desired.    
\end{proof}

By Lemma \ref{AbelianQuantDifference1}, we have
$$\log \vol B_{n,L^2} - \log \vol C_{n,L^2} = n_v \left(\log \vol B_{n,v,L^2} - \log \vol C_{n,v,L^2} \right)$$
where $n_v = 1$ if $v$ is real and $2$ if $v$ is complex and $B_{n,v,L^2}, C_{n,v,L^2}$ are the unit ball of sections in $H^0(X(\bb{C}_v),L_n)$ with $L^2$ norm $\leq 1$ for the metric of $\ovl{L}_n(\eps f)$ and $\ovl{L}_n$ respectively. Let $| \cdot |_{v,L^2}$ be the $L^2$-norm on $H^0(X(\bb{C}_v),L_n)$ induced by $\ovl{L}_n$ and $| \cdot |'_{v,L^2}$ be the $L^2$-norm induced by $\ovl{L}_n(\eps f)$.
\par 
We now proceed to bound $\log \vol B_{n,v,L^2} - \log \vol C_{n,v,L^2}$.

\begin{proposition} \label{AbelianBound1}
For all sufficiently large $n$ depending only on $X$ and $L$, we have 
$$\frac{1}{\dim V_n} \left(\log \vol B_{n,v,L^2} - \log \vol C_{n,v,L^2} \right) \geq \frac{d^n \eps }{[K:\bb{Q}]}\int_X f d \mu_{X,v} - 2 d^{n} \eps d^{-n/8} (\sup |f|) + O_r(d^n \eps^2).$$
\end{proposition}

\begin{proof}
    
We first show that 
$$\frac{1}{\dim V_n} \left(\log \vol B_{n,v,L^2} - \log \vol C_{n,v,L^2} \right) \geq \frac{d^n \eps }{[K:\bb{Q}]}\int_X f d \mu_{X,v} - d^{n} \eps d^{-n/8} (\sup |f|) + O_r(d^n \eps^2)$$
under the assumption that $f \leq 0$. This is enough to deduce the inequality in general as we may replace $f$ with $f - \sup |f|$, which increases its supremum by at most a factor of $2$.
\par 
Now let $s_1,\ldots,s_N$ be an orthonormal basis for the $L^2$-norm from $\ovl{L}_n(\eps f)$ that is orthogonal for the $L^2$-norm from $\ovl{L}_n$. Then 
$$\log \vol B_{n,v,L^2} - \log \vol C_{n,v,L^2} = \sum_{i=1}^{N} \log |s_i(x)|^2_{v,L^2}.$$
Since $| \cdot |_v = | \cdot |'_v e^{\eps f}$, we rewrite our expression as
$$\sum_{i=1}^{N} \log \int_X |s_i(x)|'^2_v e^{\eps f} d \mu_v.$$
Since $\int_X |s_i(x)|'^2_v d \mu_v = 1$, by Jensen's inequality we get
$$\sum_{i=1}^{N} \log \int_X |s_i(x)|'^2_v e^{\eps f(x)} d \mu_v \geq \int_X \sum_{i=1}^{N} |s_i(x)|_v'^2 \log (e^{\eps f(x)}) d \mu_v = \int_X \sum_{i=1}^{n} |s_i(x)|_v'^2 (d^n \eps f(x)) d \mu_v.$$
Now since $f' \leq 0$, to bound it from below it suffices to bound $\sum_{i=1}^{n} |s_i(x)|_v'^2$ from above. We now apply Proposition \ref{Bergman2} to obtain that for $n \geq O_{L}(1)$, recalling that $L_n = L^{d^n}$, we have 
$$\sum_{i=1}^{N} |s_i(x)|_v'^2 \leq \dim_{\bb{C}} H^0(X(\bb{C}_v),L_n) \left(1 + \frac{1}{d^{n/8}} \right).$$
Hence we have
$$\int_X \sum_{i=1}^{N} |s_i(x)|'^2_v (d^n \eps f(x)) d \mu \geq - \dim_{\bb{C}} H^0(X(\bb{C}_v,L_n) \left(1 + d^{-n/8} \right) d^n \eps \int_X f d \mu_v.$$
Bounding the integral by $\sup |f|$, we may now rewrite this as 
$$\frac{1}{\dim V_n} \left(\log \vol B_{n,v,L^2} - \log \vol C_{n,v,L^2} \right) \geq \frac{d^n \eps }{[K:\bb{Q}]}\int_X f d \mu_v - d^{n} \eps d^{-n/8} \sup |f|.$$
Finally, we have to convert the integral from $d \mu_v$ to $d \mu_{X,v}$. Recall that $\mu_v = (\omega'_v)^{r}$ and that $\mu_{X,v} = (\omega_v)^r$ after a suitable normalization of $\omega'_v$ and $\omega_v$. But $\omega_v$ is a positive $(1,1)$-form and the coefficients of $\omega'_v$ and $\omega_v$ differ by $O(\eps)$ and so the coefficients of its $r^{th}$ wedge power differs at most by $O_{r}(\eps c_{\omega_v,1}) = O_r(\eps)$. Thus we may change the integral of $f$ from $d\mu_v$ to $d\mu_{X,v}$ and incur an error of $O_r(d^n \eps^2)$ as desired. 
\end{proof}

We now prove our main theorem in the case of abelian varieties.

\begin{theorem} \label{AbelianQuantTheorem2}
Let $X$ be an abelian variety over a number field $K$ and $L$ a symmetric ample line bundle on $A$. Let $\h_X(x)$ be the Neron--Tate height associated to $L$. Then there exists constants $c_1 = c_1(K,L), c_3=  c_3(K,L)$ such that for any $\eps > 0$ and embedding $v: K \xhookrightarrow{} \bb{C}$, if $f: X(\bb{C}_v) \to \bb{R}$ is a smooth function, then
\begin{enumerate}
\item there exists a hypersurface $H(f,\eps)$ defined over $K$ which has degree at most $ c_1 \eps^{-8}$ with respect to $L$,
\item for all $x \in X(\ovl{K})$ with $\Gal(\ovl{K}/K)$-orbit $F_x$, if $x \not \in H(f,\eps)$ then

$$\left|\frac{1}{|F_x|} \sum_{y \in F_x} f(y) - \int f(y) d \mu_{X,v}(y) \right| \leq c_{f,3} \left(\frac{\h_{X}(x)}{\eps} + c_3 \eps \right),$$
where $c_{f,3}$ is a constant depending on the derivatives of $f$ up to order $3$ and $\mu_{X,v}$ is the Haar measure for $A$ associated to $v$.  
\end{enumerate}
\end{theorem}

\begin{proof}
Dividing $f$ by $c_{f,3}$, we may assume that $c_{f,3} \leq 2$. By Corollary \ref{Gromov2}, we obtain
$$\log C_{n,L^2} - \log C_{n,\sup} \geq -O_{\vphi}(\log d^n)\dim V_n.$$
Using Proposition \ref{AbelianBound1} along with \eqref{eq: AbelianBound2} and \eqref{eq: AbelianIneq1}, we have
$$\frac{1}{d^n \dim V_n} \left(\log \vol B_{n,\sup} - \log \vol M_n \right) \geq \frac{\eps}{[K:\bb{Q}]} \int_X f d \mu' - d^{-n/8} \eps \sup |f| + O_{\vphi,K}(d^{-n/2}).$$
We now choose the smallest $n$ such that $\eps \geq d^{-n/8}$ so that if $\int_X f d \mu_v = 0$, we get
$$\frac{1}{d^n \dim V_n} (\log \vol B_{n,\sup} - \log \vol M_n) \geq -c \eps^2$$
for some $c = c(X,K) > 0$. Now consider the adelic line bundle $\ovl{L}(\eps f + (c+1) [K:\bb{Q}] \eps^2)^n$. This will have Euler characteristic at least $d^n \dim V_n \eps^2 \geq \dim V_n + \log 2$ for $n$ sufficiently large. By Minkowski, this gives us a section $s$ of $\ovl{L}(\eps f + (c+1)[K:\bb{Q}] \eps^2)^n$ such that $\log |s|_w \leq 0$ for all places $w$.
\par 
Using this section $s$ to evaluate heights, we obtain for all $x \not \in \divv(s)$,
$$\sum_{y \in F_x} f(y) \geq -\frac{\h_{\vphi}(x)}{\eps} + (c+1)[K:\bb{Q}]\eps$$
and applying this to $-f$ gives us a corresponding upper bound. Hence we obtain a hypersurface $H(f,\eps)$ of degree at most $O_{L,K}(\eps^8)$ such that 
$$\left|\sum_{y \in F_x} f(y) - \int_X f(y) d \mu_v \right| \leq c_{f,3} \left(\frac{\h_{\vphi}(x)}{\eps} + c' \eps \right)$$
where $c'$ depends on $L$ and $K$ as desired. 
\end{proof}

As an immediate consequence, we obtain Theorem \ref{IntroAbelianGalois1}.

\begin{theorem} \label{IntroAbelianGalois1}
Let $X$ be an abelian variety over a number field $K$ and let $\h_X(x)$ denote the Neron--Tate height of $X$. Then there exists $c = c(X,K) > 0$, such that if  
$$S = \{ x \in X(\ovl{K}) \mid \h_{A}(x) \leq \frac{c}{D^8} \text{ and } [K(x):K] \leq D\},$$
then $S$ is contained in a hypersurface of degree $O_{X,K}(D^{33})$.
\end{theorem}

\begin{proof}
Fix an archimedean place $v$ and identify $X(\bb{C}_v) \simeq \bb{C}^g/\Lambda$ for a rank $2g$ lattice $\Lambda$. Let $F \subseteq \bb{C}^g$ be a fundamental domain, so that the projection map $\pi: F \to \bb{C}^g/\Lambda$ is a bijection.  We may write $F = \{c_1 v_1 + \cdots + c_{2g} v_{2g} \mid 0 \leq c_i < 1\}$ for some elements $v_i \in \bb{C}^g$. Now for a fixed positive integer $D \geq 1$, we split $F$ into $D+1$ disjoint regions, $F_1,\ldots,F_{D+1}$, where
$$F_i = \{c_1 v_1 + \cdots + c_{2g} v_{2g} \mid 0 \leq c_i < 1 \text{ for } 1 \leq i \leq 2g-1 , \frac{i-1}{D+1} \leq c_{2g} < \frac{i}{D+1} \}.$$
We now choose $f_i$ to be some test function that is supported on $F_i$ that is between $0$ and $1$ and is $1$ on 
$$F'_i = \{ c_1 v_1 + \cdots + c_{2g}v_{2g} \mid \frac{1}{4} \leq c_i \leq \frac{3}{4} \text{ for } 1 \leq i \leq 2g-1, \frac{4i-3}{4(D+1)} \leq c_{2g} \leq \frac{4i-1}{4(D+1)}.\}$$
When $v_1,\ldots,v_{2g}$ form the standard basis of $\bb{R}^{2g} \simeq \bb{C}^g$, such a $f_i$ may be chosen so that $c_{f_i,3} = O(D^3)$. Then changing basis only increases $c_{f_i,3}$ by a constant that depends only on $A$. Hence we may take $c_{f_i,3} = O_A(D^3)$ and in particular by Theorem \ref{AbelianQuantTheorem1}, there exists a hypersurface of degree $O_{A,K}(\eps^8)$ that contains all points $x \in A(\ovl{K})$ for which 
$$\left|\frac{1}{|F_x|} \sum_{y \in F_x} f_i(y) - \int f_i d \mu_v \right| \leq O_A(D^3) \left(\frac{\h_{\vphi}(x)}{\eps} + O_{A,K}(\eps).\right)$$
We know that $\int f_i d \mu_v \geq \mu_v(F_i') \geq \frac{c}{D}$ for some $c > 0$ dependign only on $A$. Choosing $\eps = c'D^{-4}$ for some constant $c' = c(A,K) > 0$ sufficiently large gives us that
$$\left|\frac{1}{|F_x|} \sum_{y \in F_x} f_i(y) - \int f_i d \mu_v \right| \leq O_A(c'D^7) \h_{\vphi}(x) + \frac{c}{3D}$$
for $x \not \in H_i$ where $H_i$ is a hypersurface of degree $O_{A,K}(D^{32})$. Hence if $\h_{\vphi}(x) \leq c'D^{-8}$ for $c' = c'(A,K) > 0$ small enough, we obtain 
$$\left|\frac{1}{|F_x|} \sum_{y \in F_x} f_i(y) - \int f_i d \mu_v \right| \leq \frac{2c}{3D} \implies \sum_{y \in F_x} f_i(y) > 0 \implies F_x \cap F_i \not = \emptyset.$$
Hence if $x \not \in \cup_{i=1}^{D+1} H_i$, we obtain that $F_x$ has non-empty intersection with $D+1$ disjoint sets, and thus $|F_x| > D$. Hence if $|F_x| < D$, we must have $x \in \cup_{i=1}^{D+1} H_i$ and so if $h_{\vphi}(x) \leq c'D^{-8}$ and $[K(x):K] \leq D$, it must be that $x$ is contained in a hypersurface of degree $D^{33}$ as desired.
\end{proof}

\section{Bergman Kernels for Negatively Curved Line Bundles} \label{sec: NegativeBergman}
We now move onto the complex geometry results necessary for the caase where the metric on our line bundle is only semipositive. In this case, after a perturbation, the line bundle $\ovl{L}(\eps f)$ will have negative curvature and thus the asymptotic expansion of Bergman kernels no longer hold. Nonetheless, Yuan \cite{Yua08} manages to overcome this technical obstruction and obtain a useful upper bound on the Bergman kernel and we will follow his approach.

\subsection{Lower bound on Bergman Kernel}
We adopt the notations in Section \ref{sec: QuantBergman}. The first step is again onto the lower bound of the Bergman kernel in which we will again follow Charles \cite{Cha15}. We endow $X$ with the Kahler metric coming from $\omega$. Let $D: L^n \to L^n \otimes \Omega^{0,1}$ be the Chern connection with $D = \partial + \ovl{\partial}$ and let $\delta,\ovl{\delta}$ be the formal adjoints of $\partial$ and $\ovl{\partial}$ respectively. We then have a Laplacian 
$$\Lap_{n,l} = \ovl{\partial} \, \ovl{\delta} + \ovl{\delta} \, \ovl{\partial}: \Omega^{0,l}(X,L^n) \to \Omega^{0,l}(X,L^n).$$
We define a Lefschetz operator $L(u) = \omega \wedge u$ and we let $\Lambda$ be the formal adjoint to $L$. We first start with computing the $(1,1)$-form $i\Theta(K_X)$ associated to the cotangent bundle. 

\begin{lemma} \label{Bergman4}
Let 
$$\omega(z) = \frac{i}{2} \sum_{1 \leq i,j \leq r} c_{i,j}(z) dz_i \wedge d \ovl{z}_j$$
in a local chart. Then if $A(z)$ denotes the positive definite matrix given by $\{c_{i,j}(z)\}$, the metric for the canonical bundle $K_X$ may be given locally by $\det A^{-1}(z)$. In particular, we may bound the coefficients of $i \Theta(K_X)$ by $O(c_{\omega,2}^{2r} \lambda_{\omega}^{-2r})$. 
\end{lemma}

\begin{proof}
In our chart, the hermitian metric on the tangent bundle in the basis $\{\frac{\partial}{\partial z_1},\ldots, \frac{\partial}{\partial z_n}\}$ is given by $A(z)$ and so the metric given on the cotangent bundle $\{dz_1, \ldots, dz_n\}$ is given by $A^{-1}(z)$. Hence the metric for the canonical bundle is given by $\det A^{-1}(z)$. Picking $dz_1 \wedge \cdots \wedge dz_n$ as our local section and computing $dd^c \log |s|$, we have to compute $-dd^c \log \det A(z)$. Taking any second derivative, say $x_i$ and $x_j$, gives us an expression of the form 
$$\frac{ \frac{\partial^2}{\partial x_i \partial x_j}(\det A(z)) \det A(z) - \frac{\partial}{\partial x_i} (\det A(z))  \frac{\partial}{\partial x_j}(\det A(z))}{ \det A(z)^2}$$
and we may bound this by $O(c_{\omega,2}^{2r} \lambda_{\omega}^{-2r})$ as desired.
\end{proof}

\begin{proposition} \label{Bergman1}
For all $n \geq O(\lambda_{\omega}^{-(2r+1)} c_{\omega,2}^{2r})$, if $t \in \Omega^{0,1}(X,L^n)$ is a smooth section, we have
$$\langle \Lap_{n,1}t,t \rangle \geq \frac{n}{2} |t|_{L^2} .$$
In particular, if $s \in C^{\infty}(X,L^n)$ which is orthogonal to the holomorphic sections $H^0(X,L^n)$, we have
$$|s|^2_{L^2} \leq \frac{2}{n} |\ovl{\partial} s|_{L^2}.$$
\end{proposition}

\begin{proof}
We may view global sections of $\Omega^{0,1}(X,L^n)$ as global $(n,1)$-forms on $K_X^{-1} \otimes L^n$. For $t \in \Omega^{0,1}(X,L^n)$, viewing it as an element $\alpha \in \Omega^{n,1}(X, K_X^{-1} \otimes L^n)$, by \cite[VII (2.1)]{Dem12} we have
$$\langle \Lap'' \alpha, \alpha \rangle \geq \int_X \langle [i \Theta(K_X^{-1} \otimes L^n), \Lambda]\alpha, \alpha\rangle d \mu$$
where $\Lap''$ is the Lpalacian for $\Omega^{n,1}(X, K_X^{-1} \otimes L^n)$. Expanding out, we have
$$[i \Theta(K_X^{-1} \otimes L^n), \Lambda] = [i \Theta(K_X^{-1}), \Lambda] + n[i \Theta(L),\Lambda].$$
By \cite[V, Proposition 5.8]{Dem12}, we know that $\langle [i \Theta(L),\Lambda] \alpha, \alpha \rangle \geq |\alpha|^2$ and so 
$$\langle n[i\Theta(L),\Lambda](\alpha),\alpha \rangle \geq n |\alpha|^2_{L^2}.$$
On the other hand, \cite[V, Proposition 5.8]{Dem12} implies that we may bound 
$$\langle [i \Theta(K_X^{-1}), \Lambda] \alpha, \alpha \rangle \geq -r \gamma |\alpha|^2$$
where $\gamma$ is an upper bound on the magnitude of the eigenvalues of $i \Theta(K_X^{-1})$ with respect to $\omega$. We may bound this by taking the ratio of the largest eigenvalue of $i \Theta(K_X^{-1})$ to the smallest eigenvalue of $\omega$ in the standard basis $z$ for our local chart. We may bound the smallest eigenvalue of $\omega$ by $\lambda_{\omega}$ and so by Lemma \ref{Bergman4}, we may bound $\gamma$ by $O(c_{\omega,2}^{2r} \lambda_{\omega}^{-(2r+1)})$. Hence if $n \geq O(c_{\omega,2}^{2r} \lambda_{\omega}^{-(2r+1)})$, we have the inequality
$$\langle [i \Theta(K_X^{-1} \otimes L^n),\Lambda] \alpha, \alpha \rangle \geq \frac{n}{2}  |\alpha|^2$$
and so
$$\langle \Lap'' \alpha, \alpha \rangle \geq \frac{n}{2}  |\alpha|^2$$
for all $n \geq O(c_{\omega,2}^{2r} \lambda_{\omega}^{-(2r+1)})$. This gives us 
$$\langle \Lap_{n,1} t , t \rangle \geq \frac{n}{2}|t|_{L^2}^2.$$
Now given $s \in C^{\infty}(X,L^n)$ that is orthogonal to $H^0(X,L^n)$, by the Hodge decomposition we know that $s = \Lap_{n,0} s'$ for some $s' \in C^{\infty}(X,L^n)$. We have $\Lap_{n,0}s' = \ovl{\delta} \ovl{\partial} s$ and so taking $t = \ovl{\partial} s$, we obtain some $t \in \Omega^{0,1}(X,L^n)$ such that $\ovl{\delta} t = s$ and $\ovl{\partial} t = 0$. Then $\Lap_{n,1} t = \ovl{\partial} s$. Let $C = \frac{n}{2}$. Since $\langle \Lap_{n,1} t , t \rangle \geq C|t|^2_{L^2}$, by Cauchy--Schwarz we get $|\Lap_{n,1} t |^2_{L^2} \geq C |t|_{L^2} |\Lap_{n,1}|_{L^2}$ and so 
$$|\ovl{\partial} s|^2_{L^2} = |\Lap_{n,1} t|^2_{L^2} \geq C |t|_{L^2} |\Lap_{n,1} t|_{L^2} \geq C \langle t, \Lap_{n,1} t \rangle = C |s|^2$$
as desired.
\end{proof}

We now prove the lower bound on the Bergman kernel.

\begin{proposition} \label{Bergman3}
Let $\ovl{L}$ be a positive line bundle with curvature form $\omega$ and let $s_1,\ldots,s_N$ be an orthonormal basis for $L^n$ for the volume form $\mu = (\omega)^{\wedge r}/r!$. Then for all $n \geq O(\lambda_{\omega}^{-(2r+1)} c_{\omega,2}^{2r} c_{h,3})^{8}$, we have
$$\sum_{i=1}^{N} |s_i(x)|^2 \geq \left(\frac{n}{2 \pi} \right)^r \left(1 - \frac{1}{n^{1/8}} \right)$$
for all $x \in X(\bb{C})$. 
\end{proposition}

\begin{proof}
As in Proposition \ref{Bergman2}, we can find a local chart of radius $O(c_{\omega,0})^{-1}$ such that our hermitian metric is given by
$$h = -\frac{1}{2} \sum_{i=1}^{r} |z_i|^2 + O(\lambda_{\omega}^{-1} c_{h,3})|z|^3.$$
In this chart, we may view a local section $t(z)$ of $\Omega^{0,1}(X,L^n)$ as $\sum_{i=1}^{r} f_i(z) d\ovl{z}_i$ where each $f_i$ is a holomorphic function. Let $g_i(z)$ be the norm of the local section of $\Omega^{0,1}(X)$ that corresponds to $d \ovl{z}_i$ in our chart. Then we may bound $|t(z)|$ by
\begin{equation} \label{eq: BergmanBound1}
|t(z)| \leq r e^{nh(z)} \cdot \max_{1 \leq i \leq r} |f_i(z)| \cdot \max_{1 \leq i \leq r} |g_i(z)| \leq e^{nh(z)} \cdot \max_{1 \leq i \leq r} |f_i(z)| \cdot O(c_{\omega,0}^{r} \lambda_{\omega}^{-(r+1)})
\end{equation}
as in our original good chart, using Lemma \ref{Bergman4} we may bound $g_i(z)$ from above by the entries of the inverse of the matrix $\{c_{i,j}(z)\}$ which has an upper bound of $O(\lambda_{\omega}^{-r} c_{\omega,0}^r)$, and then we take into account the change of coordinates with coefficients bounded by $O(\lambda_{\omega}^{-1})$.
\par 
Let $s$ be the local section that gives us the frame for which $h$ is of this form. We pick a cutoff function $\chi(z)$ in $\bb{C}^n$ that equals $1$ for $|z| \leq 1$ and vanishes for $|z| \geq 2$. We then set $\chi_n(z) = \chi( n^{2/5} z)$. Then for $n$ such that $2n^{-2/5}$ is smaller than the radius of our chart, we get that $s \chi_n$ defines for us a smooth section of $L^n$ on $X$. Let $s_n$ be the orthogonal projection of $s \chi_n$ to $H^0(X,L^n)$. By Proposition \ref{Bergman1}, for $n \geq O(\lambda_{\omega}^{-(2r+1)} c_{\omega,2}^{2r})$ we get
$$|s_n - s \chi_n|^2_{L^2} \leq \frac{2}{n} |\ovl{\partial}(s_n - s \chi_n)|^2_{L^2} = \frac{2}{n} |\ovl{\partial} (s \chi_n)|^2_{L^2}.$$
Locally in our chart, our section $s$ is simply the identity function and our function is simply $\chi_n$. Then $\ovl{\partial}(\chi_n)(z) = 0$ for $|z| \leq n^{-2/5}$ and for $|z| \geq 2n^{-2/5}$. For $z$ such that $n^{-2/5} \leq |z| \leq 2n^{-2/5}$, using \eqref{eq: BergmanBound1} we get
$$|\ovl{\partial} (s \chi_n)| = |s \ovl{\partial}(\chi_n)| = O(e^{-n^{1/5} + O(\lambda_{\omega}^{-1} c_{h,3})n^{-1/5}}) \cdot n^{1/4} \cdot  O(c_{\omega,0}^{r+1} \lambda_{\omega}^{-r})  = O(e^{-n^{1/6}})$$
since $n^{1/5} \geq O(\lambda_{\omega}^{-(2r+1)} c_{\omega,2}^{2r} c_{h,3})$. Then by \eqref{eq: BergmanVolume1}, locally we have 
$$\mu(z) = (1 + O(\lambda_{\omega}^{-1} c_{\omega,1})|z|) dV,$$
it follows that 
$$|\ovl{\partial}(s \chi_n)|_{L^2} \leq O(e^{-n^{1/6}}) \implies |s_n - s \chi_n|_{L^2} \leq O(e^{-n^{1/6}}).$$ 
Now since $s \chi_n$ is holomorphic in the polydisc of radius $n^{1/4}$, we may apply the mean value inequality along with \eqref{eq: BergmanVolume1} to obtain
$$|(s_n -s \chi_n)(0)| \leq O(e^{-n^{1/6}}) \left(1 + O(\lambda_{\omega}^{-1} c_{\omega,1}) n^{-2/5} \right) = O(e^{-n^{1/7}}).$$
We then get 
$$|(s_n - s \chi_n)(0)| \leq O(e^{-n^{1/7}}) \implies |s_n(0)| \geq 1 + O(e^{-n^{1/7}}).$$
Finally similar to \eqref{eq: Bergman2}, we verify that
$$|s \chi_n|^2_{L^2} \leq \left(1 + O(\lambda_{\omega}^{-1} c_{\omega,1}) n^{-1/4} \right)\int_{|z| \leq 2n^{-1/4}} e^{-n\vphi(z)} dV \leq \left(\frac{2 \pi}{n} \right)^r \left(1 + O(\lambda_{\omega}^{-1})c_{\omega,1}n^{-1/4} \right).$$
Hence using $s_n$ as our section, we get that 
$$\sup_{s \in H^0(X,L^n)} \frac{|s(x)|}{|s|_{L^2}} \geq \left(\frac{n}{2 \pi} \right)^r \left(1 + O(\lambda_{\omega}^{-1} c_{\omega,1})n^{-1/4} \right) \left( 1 +O \left(e^{-n^{1/7}} \right)\right) \geq \left(\frac{n}{2 \pi} \right)^r \left(1 - \frac{1}{n^{1/8}} \right)$$
for $n \geq O(\lambda_{\omega}^{-(2r+1)} c_{\omega,2}^{2r} c_{h,3})^{8}$ as desired.
\end{proof}

\subsection{Bounds on difference of archimedean volumes}

With our upper and lower bounds on the Bergman kernel for positive line bundles, we can prove the following quantitative bound on the Bergman kernel for line bundles which are not necessarily even semipositive following \cite[Proposition 2.12]{Yua08}. 

\begin{proposition} \label{Distortion1}
Let $\ovl{L}_1, \ovl{L}_2$ be two positive line bundles over $X$. Let $\{s_1,\ldots,s_N\}$ be an orthonormal basis for the metrized line bundle $\ovl{L}_1^{n_1} \otimes \ovl{L}_2^{-n_2}$, where we use the measure $\mu_1 = \frac{\omega_1^{\wedge r}}{r!}$ for the $L^2$-norm. Then if $n_i \geq O(\lambda_{\omega_i}^{-(2r+1)} c_{\omega_i,2}^{2r} c_{h_i,3})^{8}$ where $\omega_i, h_i$ are the $(1,1)$-form and hermitian metric of $\ovl{L}_i$, we have

$$\sum_{i=1}^{N} |s_i(x)|^2 \leq \left(\frac{n_1}{2 \pi} \right)^r \left(1 + \frac{1}{n_1^{1/8}} \right)\left(1 + \frac{1}{n_2^{1/8}}\right)^2.$$
\end{proposition}

\begin{proof}
We follow Yuan's proof. Fix a $x \in X(\bb{C})$ and construct an orthonormal basis for $H^0(X,\ovl{L}^{n_2}_2)$ for the measure $\mu_2 = \omega_2^{\wedge r}/r!$ such that there is exactly one section $t$ that is not vanishing at $x$. Then by Propositions \ref{Bergman2} and \ref{Bergman3} we know that 
$$|t(x)|^2 \geq \left(\frac{n_2}{\pi} \right)^{r} \left(1 - \frac{1}{n_2^{1/8}} \right) \text{ and } \sup_{z \in X(\bb{C})}|t(z)|^2 \leq \left(\frac{n_2}{\pi} \right)^r \left(1 + \frac{1}{n_2^{1/8}} \right).$$
Now let $s_1,\ldots,s_N$ be an orthonormal basis for $\ovl{L}_1^{n_1} \otimes \ovl{L}_2^{-n_2}$ and consider multiplication by $t$ to obtain a linearly independent set of sections $\{ts_1,\ldots,ts_n\}$ for $\ovl{L}_1^{n_1}$. We may choose $\{s_i\}$ such that $\{ts_i\}$ forms an orthogonal basis for the inner product on $\ovl{L}_1^{n_1}$. Applying Proposition \ref{Bergman2} tells us that 
$$\sum_{i=1}^{N} \frac{|ts_i(x)|^2}{|ts_i|_{L^2}} \leq \left(\frac{n_1}{2 \pi} \right)^r \left(1 +  \frac{1}{n_1^{1/8}} \right).$$
We also have
$$|ts_i|_{L^2} \leq |s_i|_{L^2}\sup_{z \in X(\bb{C})} |t(z)| \leq \left(\frac{n_2}{\pi} \right)^r \left(1 + \frac{1}{n_2^{1/8}} \right).$$
Hence we get
$$\sum_{i=1}^{N} |ts_i(x)|^2 \leq \left(\frac{n_1}{2 \pi} \right)^r \left(\frac{n_2}{\pi} \right)^r \left(1 + \frac{1}{n_1^{1/8}} \right) \left( 1 + \frac{1}{n_2^{1/8}} \right) $$
$$\implies \sum_{i=1}^{N} |s_i(x)|^2 \leq \left(\frac{n_1}{2 \pi} \right)^r \left(1 + \frac{1}{n_1^{1/8}} \right) \left(1 + \frac{1}{n_2^{1/8}} \right) \left(1 - \frac{1}{n_2^{1/8}} \right)^{-1}$$
where the implication follows from the lower bound on $|t(x)|^2$, as desired.
\end{proof}

We now use Proposition \ref{Distortion1} to obtain a lower bound on the difference of volumes of the unit balls for the $L^2$-norm, under a perturbation $e^{-\eps f}$ where $f$ is a smooth function. We fix a smooth function $f: X \to \bb{R}$ along with a positive line bundle $\ovl{L}$. 
\par 
First, observe that if $f$ is a smooth function satisfying $c_{f,3} \leq 2$, then the $(1,1)$-form given by $2i \del \ovl{\del} f$ is a hermitian matrix with coefficients bounded by $2$. Hence its eigenvalues are bounded from below by $-r$. We now choose some fixed positive line bundle $\ovl{L}_1$ such that $\omega_1$ is the $(1,1)$-form associated to it then its smallest eigenvalue of $\omega$ in our good charts is at least $3r$. Then $\omega_1 + 2 \del \ovl{\del} f$ is the $(1,1)$-form for $\ovl{L}_1(f)$, and we may take $\lambda_{\omega_1 + 2 \del \ovl{\del}f}$ to be $1$. We may further assume that the underlying line bundle for $\ovl{L}_1$ is $L^e$ for some positive integer $e$.  
\par 
We will set $\eps = \frac{1}{m}$ for some positive integer $m$ and set $n = mn'$. We may then express $\ovl{L}(\eps f)^n$ as 
$$\ovl{L}(\eps f)^n = \left(\ovl{L}^m \otimes \ovl{L}_1(f)\right)^{n'} \otimes \ovl{L}_1^{-n'}$$
Let $\mu'$ be the volume form associated to $\ovl{L}^m \otimes \ovl{L}_1$. Then we know that the total volume of $\mu'$ is exactly $\frac{(2 \pi)^r}{r!} (L^m \otimes L_1)^{(r)} = \frac{(2 \pi)^r}{r!}\deg(L^{m+e})$. In particular if we consider the $L^2$-norm but with $\mu'$ normalized so that it has total volume $1$, we have to multiply the norm of our sections by $\frac{(2 \pi)^r}{r!} \deg(L^{m+e})$. In particular, the leading coefficient of our bound in Proposition \ref{Distortion1} changes to
$$\deg(L^{m+e}) \frac{n'^r}{r!}.$$
Now by the Asymptotic Riemann--Roch \cite[Theorem 1.1.24]{Laz04}, we may replace this by 
$$\dim H^0(X, L^{mn' + en'}) \left( 1+ O \left(\frac{1}{n} \right) \right)$$
once $n$ and $n'$ are both large enough depending only on the ample line bundle $L$. From now on, we will let $\mu'$ be normalized to have total volume $1$ and we will let $\mu$ be the normalized volume form associated to $\ovl{L}^m$. Let $B_{n,L^2}$ and $C_{n,L^2}$ be the unit balls for the global sections of $\ovl{L}(\eps f)^n$ and $\ovl{L}^n$ under the measure $\mu'$.

\begin{proposition} \label{VolumeDifference1}
Assume that $f \leq 0$ and that $c_{f,3} \leq 2$. For 
$$n' \geq O(\lambda_{\omega}^{-(2r+1)} c_{\omega,2}^{2r} c_{h,3} m^{2r+1})^8$$ 
and also $n'$ sufficiently large depending on $L$, we have the lower bound
$$ \log \vol(B_{n,L^2}) - \log \vol(C_{n,L^2}) $$
$$\geq
n \eps \left(\dim H^0(X,L^{n+en'}) \right)  \left( 1 + \frac{1}{n'^{1/8}} \right)^5 \left(\int_X f d \mu  -  O(\eps)\right).$$
\end{proposition}

\begin{proof}
We first want to obtain an upper bound on the distortion function for the line bundle $\ovl{L}(\eps f)^n$ under the measure $\mu'$. We rewrite 
$$\ovl{L}(\eps f)^n = \left(\ovl{L}^m \otimes \ovl{L}_1(f)\right)^{n'} \otimes \ovl{L}_1^{-n'}$$
and apply Proposition \ref{Distortion1}. The curvature form for $\ovl{L}^m \otimes \ovl{L}_1$ is $m \omega + \omega_1 - dd^c f$ and since $c_{f,3} \leq 2$, we obtain that if $n' \geq O(\lambda_{\omega}^{-(2r+1)} c_{\omega,2}^{2r} c_{h,3} m^{2r+1})^8$ and both $n$ and $n'$ are sufficiently large depending only on $L$, we get the inequality
$$\sum_{i=1}^{N} |s_i(x)|^2 \leq \dim H^0(X,L^{n+en'}) \left(1 + \frac{1}{n'^{1/8}} \right)^4 \left(1 + O\left(\frac{1}{n} \right) \right) $$
$$\leq \dim H^0(X,L^{n+en'}) \left(1 + \frac{1}{n'^{1/8}} \right)^5. $$
for any orthonormal basis $\{s_i\}$ under $\mu'$, where $| \cdot |$ is the metric on $\ovl{L}(\eps f)^n$. Now pick an orthonormal basis $\{s_i\}$ such that it is orthogonal for the inner product using the metric of $\ovl{L}^n$. Then by Proposition \ref{Modules2}, we have
$$\log \vol(B_{n,L^2}) - \log \vol(C_{n,L^2}) = \sum_{i=1}^{N} \log \int_X |s_i|^2 e^{n \eps f} d \mu'.$$
Applying Jensen's, we get
$$\sum_{i=1}^{N} \log \int_X |s_i|^2 e^{n \eps f} d \mu' \geq \sum_{i=1}^{N} \int_X \log(e^{n \eps f}) |s_i|^2 d \mu' = n \eps \int_X f \left(\sum_{i=1}^{N} |s_i|^2 \right) d \mu'.$$
Since $f \leq 0$, to get a lower bound it suffices to upper bound $\sum_{i=1}^{N} |s_i|^2$ and so we obtain
$$n \eps \int_X f \left(\sum_{i=1}^{N} |s_i|^2 \right) d \mu' \geq n \eps \left(\dim H^0(X,L^n \otimes L_1^{n'} \right) \left( 1 + \frac{1}{n'^{1/8}} \right)^5 \int_X f d \mu'.$$
Finally we write $\mu'$ as the normalized volume form corresponding to the top wedge power of $\omega' = (\omega + \eps( \omega_1 + 2 i \partial \ovl{\del} f))$. Observe that $\omega'^{\wedge r}$ will have total mass equal to the top self-intersection number of $L^{n + en'}$ while $\omega^{\wedge r}$ will have total mass equal to the top self-intersection number of $L^n$. By Asymptotic Riemann--Roch \cite[Theorem 1.1.24]{Laz04}, the ratio between the two is within $1 \pm O(\eps)$ and hence we may replace $\mu'$ with $\mu$ and bound the error by $\eps O(c_{f,2}) = O(\eps)$ as desired since $|f| \leq 1$. 
 
\end{proof}

We now apply Corollary \ref{Gromov2} to convert our lower bound on volumes of $L^2$-balls to one on $\sup$-balls. Let $B_{n,\sup}$ be the unit ball with respect to the $\sup$ norm on $\ovl{L}(\eps f)$ and $C_{n,\sup}$ the unit ball for the $\sup$ norm on $\ovl{L}^n$. 

\begin{corollary} \label{VolumeDifference2} 
Assume that $f \leq 0$ and $c_{f,3} \leq 2$. Then for 
$$n' \geq O(\lambda_{\omega}^{-(2r+1)} c_{\omega,2}^{2r} c_{h,3} m^{2r+1})^8$$ 
and also sufficiently large depending on $L$, we have 
$$\frac{1}{\dim H^0(X,L^n)} \left(\log B_{n,\sup} - \log C_{n,\sup} \right) \geq n \eps \left( 1 + \frac{1}{n'^{1/8}} \right)^5 \left(\int_X f d \mu  - O(\eps) \right) .$$
\end{corollary}

\begin{proof}
We wish to apply Corollary \ref{Gromov2} for the measure $\mu'$. First, our measure is the wedge product $(\omega + \eps(\omega_1 + dd^c f))^{\wedge r}$ normalized to have volume $1$. Now again, the total degree of this is $\deg(L^{n'+ e})$ which grows at most $O(n'^{r})$. Since both are positive definite and $\lambda_{\omega}$ is a lower bound for the eigenvalues of $\omega$ and $\omega_1 + 2 \del \ovl{\del} f$ in the charts $U_{\alpha}$, we have
$$\mu' > \frac{1}{C (n')^r} \lambda_{\omega}^r dV$$
for some $C > 0$ depending on $L$. Hence we may take $c'$ in Corollary \ref{Gromov2} so that $c'^{-1} = O( (n')^r \lambda_{\omega}^{-r})$. The metric on $\ovl{L}(\eps f)$ is given by $e^{h + \eps f}$ and we may bound $\log c_{e^{h + \eps f},1}$ by $O(\log c_{e^h,1} + \log c_{f,1})$. We may then bound $\log c_2$ by $-O(c_{h,0} + c_{f,0})$ as it is bounded above by $\sup |h| + \sup |f|$. Thus by Corollary \ref{Gromov2}, the difference between the $L^2$ volumes and $\sup$ volumes may be bounded by
$$O\left( \log c_{e^h,1} + \log c_{e^{\eps f},1} + c_{h,0} + c_{\eps f,0} + \log \lambda_{\omega} + \log n \right).$$
We may then bound $\log c_{e^h,1}$ in terms of $O(c_{h,1}c_{h,0})$, $\log c_{e^{\eps f},1}$ in terms of $O(c_{f,1} c_{f,0})$. Since $c_{f,3} \leq 2$, certainly for $n \geq O(c_{h,3})^8$ we have
$$\log \vol(B_{n,\sup}) - \log \vol(B_{n,L^2}) \geq - n \left(\frac{1}{n^{1/4}} \dim H^0(X,L^n) \right).$$
We also trivially have
$$\log \vol(C_{n,\sup}) - \log \vol(C_{n,L^2}) \leq 0.$$
Now by Asymptotic Riemann--Roch \cite[Theorem 1.1.24]{Laz04}, for $n' \geq O(\eps^{-1})$ and sufficiently large depending on $L$, we have
\begin{equation} \label{eq: RiemannRoch1} 
\frac{\dim H^0(X,L^{n+en'})}{\dim H^0(X,L^n)} \leq 1 + O(\eps).
\end{equation}
Hence by Proposition \ref{VolumeDifference1}, we get for $n' \geq O(\lambda_{\omega}^{-(2r+1)} c_{\omega,2}^{2r} c_{h,3} m^{2r+1})^8$, we have
$$\frac{1}{\dim H^0(X,L^n)} \left(\log \vol(B_{n.\sup}) - \log \vol(C_{n,\sup}) \right) $$
$$\geq - n\eps \left(1 + \frac{1}{n'^{1/8}} \right)^5 \left(\int_X f d \mu -O(\eps) \right) - \frac{1}{n^{1/4}}.$$
Since $n' \geq m^{24}$, we have $\frac{1}{n^{1/4}} \leq \eps^3$ and so we obtain
$$\frac{1}{\dim H^0(X,L^n)} \left(\log B_{n,\sup} - \log C_{n,\sup} \right) \geq n \eps \left( 1 + \frac{1}{n'^{1/8}} \right)^5 \left(\int_X f d \mu  - O(\eps) \right)$$
as desired.
\end{proof}

\section{Quantitative Equidistribution for Smooth Varieties} \label{sec: QuantEquibSmooth}
Our aim in this section is to prove the main theorem of our paper. If $L$ is an ample line bundle on $X$, we fix some power $L^n$ such that $L^n$ gives an embedding of $X$ into projective space $\bb{P}^N$. We then say a subvariety $Y$ of $X$ has degree $D$ with respect to $L$ if $Y$ has degree $D$ in the embedding into $\bb{P}^N$. 

\begin{theorem} \label{QuantTheorem2}
Let $X$ be a smooth variety defined over a number field $K$ and let $\vphi: X \to X$ a polarized endomorphism of degree $d \geq 2$, i.e. there exists an ample line bundle $L$ such that $\vphi^*L \simeq L^{\otimes d}$. Then there exists constants $c_1 = c_1(K,\vphi), c_2 = c_2(\vphi), c_3 = c_3(K,\vphi)$ such that for any $\eps > 0$ and an archimedean place $v$ of $K$, if $f: X(\ovl{K}_v) \to \bb{R}$ is a smooth function, then
\begin{enumerate}
\item there exists a hypersurface $H(f,\eps)$ defined over $K$ which has degree at most $ c_1 \eps^{-c_2}$ with respect to $L$,
\item for all $x \in X(\ovl{K})$ with Galois orbit $F_x$, if $x \not \in H(f,\eps)$ then

$$\left|\frac{1}{|F_x|} \sum_{y \in F_x} f(y) - \int f(y) d \mu_{\vphi,v}(y) \right| \leq c_{f,3} \left(\frac{\h_{\vphi}(x)}{\eps} + c_3 \eps \right).$$
\end{enumerate}
\end{theorem}

Given our smooth function $f$, we first divide it by $c_{f,3}$ to get a new function $f$ which satisfies $c_{f,3} \leq 2$. It then suffices to prove that 
$$\left|\frac{1}{|F_x|} \sum_{y \in F_x} f(y) - \int f(y) d\mu_{\vphi,v}(y) \right| \leq \frac{\h_{\vphi}(x)}{\eps} + c_3 \eps$$
for functions $f$ satisfying $c_{f,3} \leq 2$ and obtain the general version by multiplying $c_{f,3}$ back in. 
\par 
By replacing $L$ with a high enough power, we may assume that $L$ embeds $X$ into $\bb{P}^N$ with $L$ being the restriction of $O(1)$, and that $\vphi$ extends to a morphism $\tilde{\vphi}: \bb{P}^N \to \bb{P}^N$ of degree $d$. We may then lift $\tilde{\vphi}$ to a homogeneous map $F: \bb{A}^{N+1} \to \bb{A}^{N+1}$. 
\par 
For each place $v$ of $M_K$, we may define a homogeneous Green's function using $F$ by
$$G_{v}(x) = \lim_{n \to \infty} \frac{\log ||F^n(x)||_v}{d^n} \text{ for } x \in \bb{A}^{N+1}(\bb{C}_v),$$
where $||(x_0,\ldots,x_{n+1})||_v = \max\{|x_0|_v,\ldots,|x_{n+1}|_v\}$. We can then use $G_v$ to define a metric on $L$ for each place $v$ by 
$$\log |s(x)|_v = \log |P(\tilde{x})|_v - G_v(\tilde{x}) \text{ for } x \in \bb{P}^N(\bb{C}_v),$$
where $\tilde{x}$ is any lift of $x$ to $\bb{A}^{N+1}(\bb{C}_v)$. Due to the homogeneity of $G_v$, our definition is independent of the lift. Restricting to $X$, this turns $L$ into a semipositive adelic line bundle $\ovl{L}$ in the sense of Zhang \cite[Section 3.5]{Yua08}, although it would not be important for us to know this. It is easy to see that its associated height on $X(\ovl{K})$ is the canonical height $\h_{\vphi}(x)$, which may also be defined as 
$$\h_{\vphi}(x) = \lim_{n \to \infty} \frac{h(\vphi^n(x))}{d^n}.$$
where $h(x)$ is the standard Weil height on $\bb{P}^N$. Let $L_n = L^{d^n}$. Then there is an induced adelic meric on $L_n$ coming from $L$ and we will denote the adelic line bundle by $\ovl{L}_n$ and the metrized line bundle on $X(\bb{C}_v)$ by $\ovl{L}_{n,v}$. 
\par 
The global sections of $L_n$, denoted by $H^0(X,L_n)$, form a $K$-vector space. We let $M_n$ be the sub $O_K$-module of sections $s$ satisfying $|s|_{v.\sup} \leq 1$ for all non-archimedean places $v$, where $\sup$ denotes the supremum over $X(\bb{C}_v)$. Then $M_n$ is a $\bb{Z}$-submodule of maximal rank of $H^0(X,L_n)$. 
\par 
We now consider $V_n = H^0(X,L_n) \otimes_{\bb{Q}} \bb{R}$. Then $V_n$ is a $\bb{R}$-vector space and we can give it a sup norm as follows. We have 
$$V_n \otimes_{\bb{R}} \bb{C} \simeq \bigoplus_{v \text { arch. }} H^0(X(\bb{C}_v),L_{n})$$
where $X(\bb{C}_v)$ denotes the complex variety obtained by embedding $K$ into $\bb{C}$ via the place $v$. Each $H^0(X(\bb{C}_v),L_{n})$ has a supremum norm coming from the metric on $L_n$ for the place $v$, and taking the supremum of each of them defines a sup norm on $V_n \otimes_{\bb{R}} \bb{C}$ and hence on $V_n$ which we will denote by $| \cdot |_{\sup}$. 
\par 
We now let $C_{n,\sup}$ be the unit ball of sections $s \in V_n$ satisfying $|s|_{\sup} \leq 1$. Fix a Haar measure on $V_n$ and let $\vol M_n$ denote the volume of the fundamental domain of the lattice $M_n$ in $V_n$. The quantity
$$\log \vol C_{n,\sup} - \log \vol M_n$$
does not depend on the choice of our Haar measure. Again, using arithmetic Hilbert--Samuel \cite[Theorem 1.4]{Zha95}, one can deduce that 
\begin{equation} \label{eq: VolumeDifference2}
\lim_{n \to \infty} \frac{1}{d^n \dim V_n} \left(\log \vol C_{n,\sup} - \log \vol M_n \right) = 0.
\end{equation}

We now regularize the metric at each archimedean place $v$ to make it smooth so that one can apply the results of Sections \ref{sec: QuantBergman} and \ref{sec: NegativeBergman}. Fix an $\eps = d^{-k}$ for some positive integer $k$. The aim is to regularize our metric so that the difference from the original metric and the new metric on $\ovl{L}_n$ is at most $d^n \eps^2$.
\par 
We do this by directly regularizing the metric of the line bundle $O(d^n)$ on $\bb{P}^r$. Fix an archimedean place $v$ and let $| \cdot |$ be the metric on $O(d^n)$. We let $| \cdot |_{\FS}$ be the Fubini-Study metric on $O(d^n)$. Then for any section $s$, the expression $\log |s(x) |^{-1} - \log | s(x)|^{-1}_{\FS}$ defines for us a continuous function $f$ on $\bb{P}^r$ satisfying $i \del \ovl{\del} f \geq -d^n \omega_{\FS}$, where $\omega_{\FS}$ is the $(1,1)$-form for the Fubini-Study metric on $O(1)$, as the current $i \del \ovl{\del}\log |s(x)|^{-1}_v + \pi [\divv(s)]$ is semi-positive. 
\par 
We now regularize the continuous function $f$ as in Proposition \ref{Regularization1}. Fix a choice of good charts for our $\bb{P}^r$. Then given $\theta$ with $|\theta| < 1$, we can regularize and obtain $f_{\theta}$ which by Lemma \ref{Regularization3}, satisfies $dd^c f_{\theta} \geq -d^n (1+O(\theta))\omega$. Furthermore, we know that $f$ is Holder continuous for some exponent $\kappa$ and constant of the form $O_{\vphi}(d^n)$. Hence taking $\theta = \eps^{2 \kappa^{-1}}$, by Lemma \ref{Regularization2} we obtain that 
$$|f_{\theta} - f| \leq d^n O_{\vphi}(\eps^2).$$
Then Proposition \ref{Regularization1} implies that 
$$c_{f_{\theta},k} \leq c_k c_{f,0} \eps^{-2 \kappa^{-1}(2r^2 + 4r + k)}.$$
We can define a metric $| \cdot |'$ on $O(d^n)$ by taking $f_{\theta} + \log | \cdot |_{\FS}$ and we hence obtain a smooth metric $| \cdot |'$ on $O(d^n)$ satisfying 
\begin{equation} \label{eq: HermitianConstants1}
c_{h',k} \leq d^n c_{h,0} \eps^{-2 \kappa^{-1}(2r^2 +4r +k)} = O_{\vphi}(d^n) \eps^{-2 \kappa^{-1}(2r^2+ 4r + k)}
\end{equation}
and 
\begin{equation} \label{eq: HermitianConstants2}
h' - h = d^n O_{\vphi}(\eps^2).
\end{equation}
where $h$ is the original canonical metric on $O(d^n)$ for the archimedean place $v$. Now restricting to $X$ and using our choice of good charts on $X$, we get a metric on $L_n$ satisfying the same inequalities, up to different constants by Lemma \ref{GoodCharts2}. 
\par 
Now for each archimedean place $v$, our line bundle $L_n$ has a smooth hermitian metric $h'_v$ that satisfies $i \del \ovl{\del} h'_v \geq -d^n O_{\vphi}(\eps^2) \omega$, where $\omega$ is the form corresponding to the Fubini-Study metric. To make it positive, we consider modfiying the hermitian metric by setting it to $(1- c \eps^2)h'_v + c d^n \eps^2 h_{\FS}$ for some constant $c = c(\vphi) > 0$, such that our new metric which we will again denote by $h'_v$, satisfies 
\begin{equation} \label{eq: Regularization1} 
i \del \ovl{\del} h'_v \geq \eps^2 d^n \omega_{\FS}.
\end{equation}

Hence for each archimedean place $v$, we obtain a metric $h'_v$ satisfying \eqref{eq: HermitianConstants1} and \eqref{eq: HermitianConstants2} and also $i \del \ovl{\del} h'_v \geq \eps^2 \omega_{\FS}$. 
\par 
Now let $\ovl{L}'_n$ be the adelic bundle with the hermitian metrics at the archimedean place being $h'$ instead of $h$ and let $C'_n$ be the unit ball of sections $s \in V_n$ satisfying $|s|_{v,\sup} \leq 1$ for all archimedean places $v$. Since $|h'-h| \leq d^n O_{\vphi}(\eps^2)$, we deduce the following lower bound.

\begin{proposition} \label{QuantProposition3}
Let $\eps = d^{-k}$ and assume that $n \geq 4k$. Then there exists $c = c(\vphi) > 0$ such that we have
$$\frac{1}{\dim V_n} ( \log \vol(C'_{n,\sup}) - \log \vol(M_n)) \geq -c (\eps^2 d^n + [K:\bb{Q}]^2 h(\alpha)).$$
\end{proposition}

\begin{proof}
Since $|h'-h| \leq d^n O_{\vphi}(\eps^2)$, we have 
$$\left|\log \vol(C_{n,\sup}) - \log \vol(C'_{n,\sup}) \right| \leq d^n (\dim V_n) O_{\vphi}(\eps^2).$$ 
As $n \geq 4k$, we have $d^{n/2} \leq d^n \eps^2$ and so combnining the above inequality with Proposition \ref{QuantProposition2}, we get our inequality as desired.
\end{proof}

\subsection{Proof of quantitative equidistribution}

We now move to estimating the volume for $L^2$-norms instead of the supremum norm. Recall that we have the isomorphism 
$$H^0(X,L_n) \otimes_{\bb{Q}} \bb{C} \simeq \bigoplus_{v \text{ arch. }} H^0(X(\bb{C}_v),L_n)$$
where we count complex conjugate pairs separately. For each archimedean place $v$, recall that we can metrize the line bundle $L_n$ with our smooth positive metric $h'_v$ which induces an hermitian inner product $\langle \, , \, \rangle _v$ on the fibers. Then if we fix a measure $\mu_v$ of total volume $1$ for each archimedean place $v$, such that the complex conjugate places share the same measure, we obtain an inner product on $V_n = H^0(X,L_n) \otimes_{\bb{Q}} \bb{R}$. We now choose our volume forms $\mu_v$. As in the paragraph before Proposition \ref{VolumeDifference1}, we fix some positive line bundle $\ovl{L}_1$ on $X(\bb{C})$, with underlying line bundle $L^e$, such that if $\omega_1$ is the $(1,1)$-form, the smallest eigenvalue in our good charts for $\omega_1$ is $\geq 3r$. Recall that $\eps = d^{-k}$. We then let $\mu'_v$ be the volume form, normalized to have measure $1$, associated to the positive line bundle $\ovl{L}_{n,v}^{d^k} \otimes \ovl{L}_1$ where $\ovl{L}_{n,v}$ is the metrized line bundle corresponding to our adelic line bundle $\ovl{L}_n$ at the place $v$. 
\par 
Since $\mu'_v$ is normalized to have volume $1$, we certainly have $|s|_{L^2,v} \leq |s|_{\sup,v}$. Thus we have the trivial bound
$$| \cdot |_{L^2} \leq [K:\bb{Q}] | \cdot |_{\sup}$$
on $V_n = H^0(X,L_n) \otimes_{\bb{Q}} \bb{R}$. Hence if we let $C'_{n,L^2}$ denote the unit ball in $V_n$ satisfying $|s|_{L^2} \leq 1$, we have the inequality
\begin{equation} \label{eq:VolumeDifference1}
\frac{1}{\dim V_n} \left(\log C'_{n,L^2} - \log C'_{n,\sup} \right) \geq -\log [K:\bb{Q}].
\end{equation}

Now let's fix an archimedean place $v$ along with a smooth function $f: X(\bb{C}_v) \to \bb{R}$ such that $c_{f,3} \leq 2$. We now perturb the metric of $\ovl{L}_n$ and $\ovl{L}'_n$ at the place $v$ by mutiplying it by $e^{-d^n \eps f}$. If $v$ has a complex conjugate pair $v'$, we also perturb the metric at $v'$ by $e^{-d^n \eps f}$. We denote our new adelic line bundle by $\ovl{L}_n(\eps f)$ and $\ovl{L}'_n(\eps f)$ respectively. Using our new metric, this induces a supremum norm on $V_n$ and also a $L^2$-norm if we use the volume forms $\mu'_v$ as we did for the original adelic line bundle $\ovl{L}_n$. 
\par 
Let $B'_{n,L^2}$ and $B'_{n,\sup}$ denote the unit ball of $V_n$ under the $L^2$-norm and $\sup$ norm of $\ovl{L}'_n(\eps f)$ respectively. Our aim is to get a lower bound of $\log \vol B'_{n,\sup} - \log \vol M_n$. We first bound $\log \vol B'_{n,L^2} - \log \vol C'_{n,L^2}$ from below. 

\begin{lemma} \label{GlobalVolumeDifference1}
Let $n_v = 2$ if $v$ is complex and $1$ if $v$ is real. Then 
$$\log \vol B'_{n,L^2} - \log \vol C'_{n,L^2} = n_v \left(\log \vol B_{n,v,L^2} - \log \vol C'_{n,v,L^2} \right)$$
where $B'_{n,v,L^2}$ is the unit ball of sections with $L^2$-norm $\leq 1$ for $H^0(X(\bb{C}_v),L_n)$ with the metric from $\ovl{L}'_n(\eps f)$ and similarly for $C'_{n,v,L^2}$. 
\end{lemma}

\begin{proof}
This is exactly the same as Lemma \ref{AbelianQuantDifference1}.
\end{proof}
We are now in a position to apply the volume bounds from Section \ref{sec: NegativeBergman}. Recall again that $\eps = d^{-k}$. Let $\mu_v$ be the volume form associated to the positive line bundle $\ovl{L}'_{n,v}$, normalized to have volume $1$.

\begin{proposition} \label{GlobalVolumeDifference2}
Assume that $f \leq 0$. Let $\kappa$ be a Holder exponent for the homogeneous Green's function $G_{v}$ associated to $F$. Then for all $n \geq O_r(\kappa^{-1} k) + O_{\vphi}(1)$, we have
$$\frac{1}{d^n \dim V_n} \left( \log \vol B'_{n,v,L^2} - \log \vol C'_{n,v,L^2} \right) \geq \frac{\eps}{[K:\bb{Q}]} \left(\int_X f d \mu_{v} + O_{\vphi}(\eps) \right).$$
\end{proposition}

\begin{proof}
We wish to apply Proposition \ref{VolumeDifference1}. We certainly have the conditions $f \leq 0$ and $c_{f,3} \leq 2$. Let $\ovl{L}_{n,v} = \ovl{L}_v^{d^n}$ where the underlying line bundle of $\ovl{L}_v$ is $O(1)$ on $X$. Let $\omega_v$ be the $(1,1)$-form corresponding to $\ovl{L}_v$ and let $h_v$ be the hermitian metric on $\ovl{L}_v$. We have $m = \eps^{-1} = d^k$ and $n' = d^{n-k}$ and so we require    
$$d^{n-k} \geq O_r(\lambda_{\omega_v}^{-(2r+1)} c_{\omega_v,2}^{2r} c_{h_v,3} d^{(2r+1)k})^8$$
to apply Proposition \ref{VolumeDifference1}. By \eqref{eq: Regularization1}, since $dd^c h_v = \frac{1}{d^n} dd^c h'_v$, we know that $\omega_v \geq \eps^2 \omega_{\FS}$ and so $\lambda_{\omega} = O(\eps^{-2}) = O(d^{2k})$. Next, by \eqref{eq: HermitianConstants1}, we know that $c_{\omega_v,2} \leq O(c_{h,4}) = O_{\vphi}(\eps^{-O_r(\kappa^{-1})}) = O_{\vphi}(d^{O_r(\kappa^{-1}k)})$ and we may also bound $c_{h,3}$ similarly. Thus it suffices to have
$$d^{n-k} \geq d^{O_r(k)} O_{\vphi}(d^{O_r(\kappa^{-1}k)}) = O_{\vphi}(d^{O_r(\kappa^{-1}k)}).$$
Hence to apply Proposition \ref{VolumeDifference1}, it suffices to have $n \geq O_r(\kappa^{-1}k) + O_{\vphi}(1)$. Proposition \ref{VolumeDifference1} then tells us that
$$\frac{1}{d^n } \left(\log \vol B_{n,v,L^2} - \log \vol C_{n,v,L^2} \right) $$
$$\geq \eps \dim H^0(X(\bb{C}_v), L_n \otimes L_1^{n'}) (1 + \eps^2) \left(\int_X f d \mu_v + O_{\vphi}(\eps) \right).$$
Now similar to \eqref{eq: RiemannRoch1}, as $L_1 = L^e$ for some fixed positive integer $e$, we know that for $n$ sufficiently large depending only on $L$, we have
$$\dim H^0(X(\bb{C}_v), L_n \otimes L_1^{n'}) = (1+O(\eps)) \dim H^0(X(\bb{C}_v), L_n) = \frac{1 + O(\eps)}{[K:\bb{Q}]} \dim V_n.$$
Hence we obtain
$$\frac{1}{d^n \dim V_n} \left(\log \vol B'_{n,v,L^2} - \log \vol C'_{n,v,L^2} \right) \geq \frac{\eps}{[K:\bb{Q}]} \left( \int_X f d \mu_v + O_r(\eps) \right) $$
as desired. 
\end{proof}

Now by the same argument as in Corollary \ref{VolumeDifference2} applied to each archimedean place $v$, using Gromov's inequality we obtain for $n \geq O_r(\kappa^{-1} k) + O_{\vphi}(1)$ that 
\begin{equation} \label{eqref: GlobalVolumeBound1}
\log \vol B'_{n,\sup} - \log \vol B'_{n,L^2} \geq -d^n \eps^2 \dim V_n.
\end{equation}

Let $B_{n,\sup}$ be the unit ball of $V_n$ under the metric coming from $\ovl{L}(\eps f)$. Putting it all together, we obtain the following lower bound on $\log \vol B_{n,\sup} - \log \vol M_n$. 

\begin{proposition} \label{GlobalVolumeDifference3}
Assume that $f \leq 0$ and let $\kappa$ be a Holder exponent for the Green's function $G_v$ associated to $F$. Then for all $n \geq O_r(\kappa^{-1}k) + O_{\vphi}(1)$, we have
$$\frac{1}{d^n \dim V_n} \left(\log \vol B_{n,\sup} - \log \vol M_n \right) \geq \frac{\eps n_v}{[K:\bb{Q}]} \int_X f d\mu_{\vphi,v} - O_{\vphi}\left(\eps^2 + \frac{[K:\bb{Q}]^2 h(\alpha)}{d^n} \right)$$
where $\mu_{\vphi,v}$ is the equilibrium measure for $\vphi$ at the archimedean place $v$. 
\end{proposition}

\begin{proof}
Combining Propositions \ref{QuantProposition3}, \ref{GlobalVolumeDifference2} along with the inequalities from \eqref{eq:VolumeDifference1} and \eqref{eqref: GlobalVolumeBound1}, we obtain
$$\frac{1}{d^n \dim V_n} (\log \vol B'_{n,\sup} - \log \vol M_n) \geq \frac{\eps n_v}{[K:\bb{Q}]} \int_X f d \mu_v - O_{\vphi}\left(\eps^2 + \frac{[K:\bb{Q}]^2 h(\alpha)}{d^n} \right).$$
Now by \eqref{eq: HermitianConstants2}, we know that $h'$ and $h$ differ by $d^n O_{\vphi}(\eps^2)$ and so we have
$$\frac{1}{d^n \dim V_n} \left( \log \vol B_{n,\sup} - \log \vol B'_{n,\sup} \right) \geq \frac{\eps n_v}{[K:\bb{Q}]} \int_X f d \mu_v - O_{\vphi} \left(\eps^2 + \frac{[K:\bb{Q}] h(\alpha)}{d^n} \right).$$
It now suffices to obtain a bound for
$$\left(\int_X f d \mu_v - \int_X f d \mu_{\vphi,v} \right).$$
Let $\{U_{\alpha}\}_{\alpha \in I}$ be our choice of good charts, with $I$ being a finite set. Recall that each $U_{\alpha}$ is isomorphic to the polydisc $\bb{D}_3$ By using a partition of unity, we may assume that $f$ is supported within $\bb{D}_{3/2}$ of one of our charts $U_{\alpha}$. Let $u = i \frac{h_{\alpha}}{d^n}$ and $u' = i \frac{h'_{\alpha}}{d^n}$. Then $u$ and $u'$ are p.s.h. functions and locally within $\bb{D}_2$, we know that 
$$d \mu_v = c (dd^c u)^{\wedge r} \text{ and } d \mu'_v = c (dd^c u')^{\wedge r}$$
for some constant $c$ depending only on $r$.
We write 
$$d \mu_v - d \mu'_v = c\sum_{i=0}^{r-1} \left((dd^c u)^{\wedge (r-i)} \wedge (dd^c u')^{\wedge i} - (dd^c u)^{\wedge (r-i-1)} \wedge (dd^c u')^{\wedge (i+1)} \right).$$
We may rewrite each term in the sum as 
$$dd^c \left((u-u') \wedge (dd^c u)^{\wedge (r-i-1)} \wedge (dd^c u')^{\wedge i} \right)$$
where this expression makes sense as both $u$ and $u'$ are p.s.h. 
We now apply the Chern--Levine--Nirenberg inequality \cite[Chapter 3, 3.3]{Dem12} to $(dd^c u)^{\wedge r-i-1} \wedge (dd^c u')^{\wedge i}$ to obtain that 
$$||(u'-u)(dd^c u)^{\wedge (r-i-1)} \wedge (dd^c u')^{\wedge i}||_{\bb{D}_{3/2}} $$
$$\leq C |u|^{r-i-1}_{L^{\infty}(\bb{D}_2)} |u'|^{i}_{L^{\infty}(\bb{D}_2)} |u-u'|_{L^{\infty}(\bb{D}_2)}$$
for some constant $C > 0$ independent of $u$ and $u'$. But since $h'-h = d^n O_{\vphi}(\eps^2)$ we may bound this by $c_{h,0}^{r} \cdot O_{\vphi}(\eps^2)$ and hence we obtain that
$$||(u'-u)(dd^c u)^{\wedge (r-i-1)} \wedge (dd^c u')^{\wedge i}||_{\bb{D}_{3/2}} \leq O_{\vphi}(\eps^2).$$
Finally as $f$ is of compact support inside $\bb{D}_{3/2}$, we have
$$\int f dd^c \left((u-u') \wedge (dd^c u)^{\wedge (r-i-1)} \wedge (dd^c u')^{\wedge i} \right) = \int (dd^c f) \left((u-u') \wedge (dd^c u)^{\wedge (r-i-1)} \wedge (dd^c u')^{\wedge i} \right) $$
$$\leq O_{\vphi}(\eps^2 \cdot c_{f,2}) \leq O_{\vphi}(\eps^2)$$
as $c_{f,2} \leq 1$. Hence we may replace $d \mu_v$ with $d \mu_{\vphi,v}$ at the cost of an $O_{\vphi}(\eps^2)$ error. 
\end{proof}

We now prove Theorem \ref{QuantTheorem2}. 

\begin{proof}[Proof of Theorem \ref{QuantTheorem2}]
Given a smooth $f: X(\bb{C}_v) \to \bb{R}$, we first add a constant to it so that $\int_X f d \mu_{\vphi,v} = 0$. We may also divide $f$ by $c_{f,3}$ and so we may assume that $c_{f,3} \leq 2$. Given $\eps > 0$, we round it down to the nearest $d^{-k}$. This changes our constants by at most a bounded power of $d$. It now suffices to show the existence of $c_1,c_2,c_3 > 0$ such that for $x \in X(\ovl{K})$ outside of a hypersurface of degree at most $c_1 \eps^{-c_2}$, we have
$$\frac{1}{|F_x|} \sum_{y \in F_x} f(y) \geq -\left(\frac{\h_{\vphi}(x)}{\eps} + c_3 \eps \right)$$
for we can deduce the same inequality for $-f$ to get the corresponding upper bound. Let $B_{n,\sup}$ be the unit ball of $V_n$ for the adelic line bundle $\ovl{L}_n(\eps f)$. Then for $n \geq O_r(\kappa^{-1} k) + O_{\vphi}(1)$, we claim that 
$$\frac{1}{d^n \dim V_n} (\log \vol B_{n,\sup} - \log \vol M_n) \geq O_{\vphi}\left(\eps^2 + \frac{[K:\bb{Q}]h(\alpha)}{d^n} \right)$$
Indeed, we set $f' = f - c$ for some constant $c > 0$ such that $f' \leq 0$ and apply Proposition \ref{GlobalVolumeDifference3}. Scaling back $f'$ to $f$ then replaces $\int_X f' d\mu_{\vphi,v}$ with $\int_X f d\mu_{\vphi,v}$. Now for some suitable $c_3 = c_3(K,\vphi)$, the corresponding unit ball $B'_{n,\sup}$ for the metric on $\ovl{L}(\eps f + c_3 \eps^2 )$ satisfies 
$$\frac{1}{d^n \dim V_n} (\log \vol B'_{n,\sup} - \log \vol M_n) \geq \eps^2 \implies \log \vol B'_{n,\sup} - \log \vol M_n \geq \eps^2 d^n \dim V_n.$$
As $d^n \eps^2 \geq \log 2$, we may apply Minkowski's theorem to conclude that there exists a section $s$ of $H^0(X,L_n)$ such that $|s|_{\sup,v} \leq 1$ for all places $v$ of $\ovl{L}_n(\eps f + c_3 \eps^2)$. 
\par 
For $x \not \in \divv(s)$, we use our small section $s$ to compute the height of $x$ with respect to $\ovl{L}_n(\eps f + c_3 \eps^2)$, which we will denote by $h_L$. As our section $s$ is small, we have $h_L(x) \geq 0$. On the other hand, we have
$$h_L(x) = \h_{\vphi}(x) + \frac{\eps }{|F_x|}\sum_{y \in F_x} f(y) + c_3 \eps^2.$$
Dividing by $\eps$, we obtain that
$$\frac{1}{|F_x|} \sum_{y \in F_x} f(y) \geq \frac{-\h_{\vphi}(x)}{\eps} - c_3 \eps$$
which is what we wanted, as desired. Our inequalities holds as long as $x \not \in \divv(s)$, which is a hypersurface defined over $K$ of degree $d^n$ intersected with $X \subset \bb{P}^N$. Since we may take $n = O_r(\kappa^{-1} k) + O_{\vphi}(1)$, this is a hypersurface of degree $c_1 \eps^{-c_2}$ where $c_1 = c_1(\vphi)$ and $c_2 = c_2(\vphi)$ as desired.
\end{proof}

\section{Quantitative Equidistribution of Periodic Points on Surfaces}
We now give our first application of Theorem \ref{QuantTheorem2}. Let $X$ be a surface, $\vphi: X \to X$ be a polarized endomorphism of degree $d \geq 2$ and let $\Per_n$ be the set the points of period $n$. When $X = \bb{P}^2$ Briend--Duval \cite{BD01} has shown that the discrete measure $\frac{1}{|\Per_n|} \sum_{x \in \Per_n} \delta_x$ converges to the equilibrium measure $\mu_{\vphi}$. In general, equidistribution of periodic points has been shown by Dinh--Nguyen--Truong \cite{DNT15}. 
\par 
When $\vphi$ is defined over a number field $K$, we will prove that an exponential rate of convergence holds, answering \cite[Problem 7]{DS17} in the case where $\vphi$ is defined over a number field.  
\par 
To do so, we will bound the number of elements of $\Per_n$ that can lie on a geometrically irreducible degree $e$ curve $C$ in terms of $e$. By degree, we mean the induced degree after embedding $X \xhookrightarrow{} \bb{P}^N$ so that $\vphi$ extends to $F: \bb{P}^N \to \bb{P}^N$. We will split into two cases, the first is when $C$ is periodic of period $n$ and the second is when $C$ is not. We handle the second case first.

\begin{proposition} \label{PeriodicCurve1}
Let $C \subseteq X$ be a degree $e$ curve, irreducible over $\bb{C}$, such that $f^n(C) \not = C$. Then there are at most $e^2 d^n$ many points of period $n$ that lie on $C$.  
\end{proposition}

\begin{proof}
Let $L = O(1)$ on $\bb{P}^N$. By the projection formula, we know that 
$$(F^n)^*L \cdot C = F^n_*C \cdot L.$$
Since $(f^n)^*L = O(d^n)$, the LHS is $ed^n$. The RHS is at least $\deg(\vphi^n(C))$ and so we get that $\deg(\vphi^n(C)) \leq ed^n$. Now if $x$ is a periodic point of period $n$ that lies in $C$, then it also lies in $\vphi^n(C)$ and hence lies in the intersection $C \cap \vphi^n(C)$. Since $C$ is irreducible over $\bb{C}$ and is not equal to $\vphi^n(C)$, by Bezout's we have 
$$|C \cap \vphi^n(C)| \leq \deg(\vphi^n(C)) \cdot \deg(C) \leq e^2 d^n$$
and hence there are at most $e^2 d^n$ many points of period $n$ that lie on $C$. 
\end{proof}

We now handle the other case, which is when $C$ is periodic of period $n$. Let $C'$ be a smooth model of $C$ and so $\vphi^n$ induces a birational map $C' \to C'$. It then extends to an endomorphism $\tilde{\vphi^n}: C' \to C'$. 

\begin{proposition} \label{PeriodicCurve2}
The number of fixed points of $\tilde{\vphi^n}$ on $C'$ is at most $d^n + 2d^{n/2} + 1$. 
\end{proposition}

\begin{proof}
If $C'$ is of genus $g \geq 2$, it follows that $\tilde{\vphi^n}$ is of finite order. Then there exists some $m$ such that $f^{nm}$ induces the identity map on $C$ but this is impossible as this would imply that $\vphi^{nm}$ on $X$ has infinitely many fixed points. Hence $C'$ is either isomorphic to $\bb{P}^1$ or to an elliptic curve $E$. 
\par 
If $C'$ is isomorphic to $\bb{P}^1$ or $\bb{E}$, then by Lemma 2.2 of \cite{MMSSZZ22}, we know that the first dynamical degree of $\tilde{\vphi^n}$ is at most the first dynamical degree of $\vphi^n$, which is $d^n$. We now apply Proposition 6.2 of \cite{MMSSZZ22} to conclude that there are at most $d^n + 2d^{n/2} + 1$ many fixed points for $\tilde{\vphi^n}$ as desired. 
\end{proof}

Putting both propositions together, we obtain the following bound on not necessarily irreducible curves.

\begin{proposition} \label{PeriodicCurve3}
Let $C$ be a not necessarily irreducible curve of degree $e$. Then there are at most $3 e^2 d^n $ many periodic points of period $n$ that lie on $C$. 
\end{proposition}

\begin{proof}
We write $C$ as an union of irreducicble curves over $\bb{C}$, say $C_1,C_2,\ldots,C_m$. Then if $\deg(C_i) = e_i$, we have $\sum_{i=1}^{m} e_i = e$. Let $C_1,\ldots,C_{k}$ be curves that are not periodic of period $n$ and let $C_{k+1},\ldots,C_m$ be those that are periodic of period $m$. Then $\sum_{i=1}^{k} e_i \leq e-(m-k)$. Applying Proposition \ref{PeriodicCurve2} to $C_{k+1},\ldots,C_m$, we obtain that the number of periodic points of period $n$ on each of them is at most $d^n + 2d^{n/2} + 1$ and in total there is at most $(m-k)(d^n + 2d^{n/2} + 1)$ periodic points of period $n$ among $C_{k+1},\ldots,C_m$.
\par 
For $C_1,\ldots,C_k$, we apply Proposition \ref{PeriodicCurve1} to obtain that they have at most $d^n \sum_{i=1}^{k} e_i^2$ many periodic points of period $n$. Since $\sum_{i=1}^{k} e_i^2 \leq (\sum_{i=1}^{k} e_i)^2$, this is at most $d^n (e-(m-k))^2$. Putting it together, we have at most 
$$d^n (e-(m-k))^2 + (m-k) (d^n + 2d^{n/2} + 1) \leq e^2 (d^n + 2d^{n/2} + 1) \leq 3e^2 d^n$$
periodic points of period $n$ that lie on $C$. 
\end{proof}

We can now prove our exponential rate of convergence for periodic points of $\vphi$.

\begin{theorem} \label{PeriodicCurve4}
Let $X$ be a smooth projective surface over a number field $K$ and $\vphi: X \to X$ be a polarized endomorphism of degree $d \geq 2$. Then there exists constants $C > 0$ and $\lambda > 1$, depending on $\vphi$, such that if $v$ is an archimedean place of $K$, then for any smooth function $f: X(\bb{C}_v) \to \bb{R}$, we have
$$\left|\frac{1}{|\Per_n|} \sum_{x \in \Per_n} f(x) - \int f d \mu_{\vphi,v} \right| \leq C c_{f,3}\lambda^{-n}.$$
\end{theorem}

\begin{proof}
Let $f$ be a given smooth function and let's assume that $K$ is the minimal field of definition for $\vphi$. Let $c_1,c_2,c_3$ be the constants in Theorem \ref{QuantTheorem2} for $\vphi$ and $K$. We set $\eps = d^{n/4c_1}$. We write $\Per_n$ as an union of $\Gal(\ovl{K}/K)$-orbits $F_1,\ldots,F_m$. Now by Theorem \ref{QuantTheorem2}, there exists a curve $C_n$ defined over $K$, of degree $c_1 \eps^{-c_2} = c_1 d^{n/4}$, such that if $x$ is a periodic point that doesn't lie on it, we have
$$\left|\frac{1}{|F_x|} \sum_{y \in F_x} f(y) - \int f d \mu_{\vphi,v} \right| \leq c_3 \eps  = c_3 d^{-n/4c_2}.$$
Now by Proposition \ref{PeriodicCurve3}, we know that $C_n$ contains at most $3c_1^2 d^{3n/2} $ many periodic points of period $n$. Let $F_1,\ldots,F_k$ be the $\Gal(\ovl{K}/K)$-orbits in $\Per_n$ that do not lie on $C_n$ and let $F$ be their union. Then we also have
$$\left|\frac{1}{|F|} \sum_{y \in F} f(y) - \int f d \mu_{\vphi,v} \right| \leq c_3 c_{f,3} d^{-n/4c_2}.$$
Now there are at most $3c_1^2 d^{3n/2}$ remaining periodic points of period $n$, and by \cite[Theorem 1.1]{DZ23}, we know that $|\Per_n| \geq \alpha d^{2n}$ for some constant $\alpha > 0$ depending on $\vphi$. Hence the remaining points can only add an error of $3 \alpha^{-1} c_1^{2} c_{f,0} d^{-n/2}$. As we may take $c_1,c_2 \geq 1$, this is $\leq 3 \alpha^{-1} c_1^{2} c_{f,3} d^{-n/4c_2}$ and so we obtain
$$\left|\frac{1}{|\Per_n|} \sum_{x \in \Per_n} f(x) - \int f d\mu_{\vphi,v} \right| \leq 3 \alpha^{-1} c_1^{2} c_{f,3} C_3d^{-n/4c_2}.$$
We may then take $\lambda = d^{1/4c_2}$ and $C = 3 \alpha^{-1} c_1^2 c_{f,3} C_3$ as desired.
\end{proof}

\section{Bound on Preperiodic Points of Bounded Degree}
We now turn into our next application of Theorem \ref{QuantTheorem2}. Let $X$ be a smooth projective variety over a number field $K$ and $\vphi: X \to X$ a polarized endomorphism of degree $d \geq 2$. Our first goal is to control the geometry of preperiodic points of $\vphi$ that are of degree $D$ over our base field $K$. We will refer to this set as $\Prep_{\vphi,D}$. 
\par 
Fix an archimedean place $v$ of $M_K$ along with a choice of good charts $\cal{U}_I$ for $X(\bb{C}_v)$. Then by \cite{DNS10}, we know that the equilibrium measure $\mu_{\vphi,v}$ is locally moderate and in particular, we can find constants $C,\kappa > 0$ such that $\mu_{\vphi,v}(D(x,r)) \leq Cr^{\kappa}$ for all sufficiently small $r$, where we view $x$ as lying inside $\bb{D}_1$ for some chart $U_i \simeq \bb{D}_3$. 

\begin{lemma} \label{Measure1}
There exists a chart $U_i \simeq \bb{D}_3$ such that $\mu(\bb{D}_1) > 0$.    
\end{lemma}

\begin{proof}
This is clear as the union of the polydiscs $\bb{D}_1$ cover $X$. 
\end{proof}

Now let's fix a positive integer $D > 0$. We first show that we can find $D$ cubes, each having a non-negligible amount of measure, and also each having a neighbourhood that are pairwise disjoint from each other. Using Lemma \ref{Measure1}, we may assume that we are inside $\bb{D}_2$ where $\mu_{\vphi,v}(\bb{D}_1) = \delta > 0$. If $R$ is a cube in $\bb{C}^r \simeq \bb{R}^{2r}$, we let $2R$ denote the cube that shares the same center as $R$, but with lengths doubled.

\begin{proposition} \label{Cubes1}
Fix a positive integer $n$. Then there exists a constant $c = c(\vphi) > 0$ such that for any positive integer $D \geq 1$, there are $D$ cubes $\{R_i\}_{i=1}^{D}$, each of length $cD^{-2 \kappa}$, such that $\mu(R_i) \geq \frac{1}{2} c D^{-2 r \kappa}$ and that the cubes $\{nR_i\}_{i=1}^{D}$ are all pairwise disjoint.  
\end{proposition}

\begin{proof}
As noted before, we can find $C$ and $\kappa$ such that if $R$ is a cube of length $\eps$, then $\mu_{\vphi}(R) \leq  C \eps^{\kappa}$. Also, we may find some chart locally isomorphic to $\bb{D}_2$ such that $\mu(\bb{D}_1) = \delta > 0$. We first view $\bb{D}_1$ inside the cube of radius $2$ centered at $0$, and then partition this cube into disjoint cubes of length $c' D^{-2/\kappa}$ for some $c > 0$. Then there are $(c' D^{-2/\kappa})^{-r}$ many such cubes and hence among the cubes covering $\bb{D}_1$, there must be one with measure at least $\delta(c' D^{-2/\kappa})^r$. 
\par 
Let $R_1$ be this cube. Then $nR$ is a cube of length $nc' D^{-2/\kappa}$ and so $\mu_{\vphi}(nR) \leq nc'C D^{-2}$. Thus if $c' = \frac{\delta}{nC}$, the remaining cubes which do not intersect $nR$ will have total measure at least $\delta (1- D^{-2})$ left. Choosing $R_2$ from the remaining cubes will ensure that $2R_i$ remains disjoint with each other. We may choose $R_2$ such that $\mu(R_2) \geq \delta (1- D^{-2})(c' D^{-2/\kappa})^r$ and repeating this argument, we may choose $R_k$ such that 
$$\mu(R_k) \geq \delta (1- k D^{-2})(c D^{-2/\kappa})^r \geq (1 - D^{-1})(c D^{-2/\kappa})^r \geq \frac{1}{2}(\delta c')(D^{-2/\kappa})^r.$$
Hence taking $c = \delta c' = \frac{\delta^2}{nC}$ gives us the result we want.
\end{proof}

We now apply Theorem \ref{QuantTheorem2} to obtain the existence of a hypersurface $H$ of controlled degree that contains $\Prep_{\vphi,D}$. 

\begin{theorem} \label{Galois1}
There exists a constant $e = e(\vphi) > 0$ such that for any positive integer $D > 0$, there exists a hypersurface $H$ of degree at most $O_{\vphi,K}(D^e)$, such that $\Prep_{\vphi,D} \subseteq H$.  
\end{theorem}

\begin{proof}
Given $D \geq 2$, we apply Proposition \ref{Cubes1} to find $D$ cubes $\{R_i\}_{i=1}^{D}$, of length $cD^{-2\kappa}$, and having measure $\mu(R_i) \geq \frac{1}{2} cD^{-2r\kappa}$ and such that $2R_i$ are all disjoint. For each $i$, we let $f_i$ be a non-negative smooth function that is identically $1$ on $R_i$ and zero outside of $2R_i$. If $f$ is a smooth function that satisfies this property when $R$ is the unit cube, we may take  $f_i(x)$ to be a translation of $f(c^{-1} D^{2\kappa}x) $. In particular, we have 
$$c_{f_i,3} = O_{\vphi}(D^{6 \kappa}).$$
Now by Theorem \ref{QuantTheorem2}, for $\eps > 0$ and $x \in \ovl{\bb{Q}}$ with $\h_{\vphi}(x) = 0$, there exists a hypersurface of degree $H(f_i,\eps)$ of degree $O_{\vphi}(\eps^{-c_2})$ such that 
$$\left|\frac{1}{|F_x|} \sum_{y \in F_x} f(y) - \int f_i(y) d \mu_{\vphi}(y) \right| \leq c_{f_i,3} c_3 \eps$$
for all $x \not \in H(f_i,\eps)$. 
Since $\mu_{\vphi}(R_i) \geq \frac{1}{2} cD^{-2r\kappa}$, we have 
$$\int f_i(y) d \mu_{\vphi}(y) \geq \frac{1}{2} cD^{-2r\kappa}.$$
Thus if we take $\eps$ such that $c_{f_i,2} c_3 \eps < \frac{1}{2} cD^{-2r \kappa}$, it follows that for $x \not \in H(f_i,\eps)$, we have $f_i(y) > 0$ for some $y \in F_x$. Hence $y \in 2R_i$ for some $y \in F_x$. 
\par 
Repeating this for each $R_i$ gives us that outside a hypersurface of degree $O_{\vphi}(D \eps^{-c_1})$, we must have $y \in 2R_i$ for some $y \in F_x$ for each $i$. But since $2R_i$'s are disjoint from each other, it must be that $|F_x| \geq D$ and so the union of $H(f_i,\eps)$ must contain all preperiodic points of algebraic degree $< D$. Now it suffices to note that $\eps$ may be taken as $D^{-O(r \kappa)}$ and so our hypersurface is of degree at most $O_{\vphi,K}(D^e)$, where $e$ depends on $\vphi$ as desired.
\end{proof}

As remarked in the introduction, when $r = 1$, Theorem \ref{Galois1} implies that there are at most $O_{\vphi,K}(D^e)$ many preperiodic points of degree $\leq D$ over $K$. This is different from Baker's results \cite{Bak06} in that Baker requires us to fix a single extension $L$ of degree $D$ whereas we can vary over all extensions of degree $D$. We will show that in the dimension one case, we can also deduce this from Favre--Rivera-Letelier's quantitative equidistribution theorem. 

\begin{proposition} \label{GaloisDimOne1}
Let $\vphi: \bb{P}^1 \to \bb{P}^1$ be a rational map defined over a number field $K$ of degree $d \geq 2$. Then there exists $e = e(\vphi)$ such that there are at most $O_{\vphi,K}(D^e)$ many preperiodic points of degree $\leq D$ over $K$. 
\end{proposition}

\begin{proof}
We give a sketch. By Proposition \ref{Cubes1}, we can find $D+1$ cubes $\{R_i\}$ such that $2R_i$ are pairwise disjoint, each having measure $\frac{1}{2} cD^{-2 r \kappa}$. Now group each preperiodic point of degree $\leq D$ over $K$ with its Galois orbit. Each Galois orbit has at most $D$ elements and so there is one cube $R_i$ such that it does not lie in $2R_i$. Now given $cD^e(D+1)$ many Galois orbits, there must exist one $i$ such that $cD^{e}$ of the Galois orbits all do not intersect $2R_i$. Applying \cite[Theoreme 3]{FRL06} with a non-negative function $f$ which is identically $1$ on $R_i$ and supported on $2R_i$ then gives us a bound on $c$ and $e$ as desired.      
\end{proof}

Sometimes it is more desirable to obtain an actual numerical bound on the size $\Prep_{\vphi,D}$ rather than just controlling its geometry. We show that when $X = \bb{P}^2$, under some assumptions it is possible to obtain a bound on $|\Per_{\vphi,D}|$, the set of periodic points of $\vphi$ that are of degree $D$ over the base field $K$, that is polynomial in $D$.  
\par 
Let $\vphi$ be an endomorphism on $\bb{P}^2$ of degree $d \geq 2$ defined over a number field $K$. We will further assume that there are only finitely many periodic subvarieties of $\vphi$ in $\bb{P}^2$. Heuristically, this is expected to hold for a general endomorphism $\vphi$, although it has only been shown to hold for a very general endomorphism \cite{Xie24}.

\begin{theorem} \label{Galois2}
Let $\vphi: \bb{P}^2 \to \bb{P}^2$ is an endomorphism of degree $d \geq 2$ defined over a number field $K$. Further assume that there are only finitely many periodic curves for $\vphi$. Then there exists $e = e(\vphi) > 0$ such that $|\Per_{\vphi,D}| \leq O_{\vphi,K}(D^e)$.
\end{theorem}

\begin{proof}
Given a positive integer $D \geq 2$, by Theorem \ref{Galois1} we can find a hypersurface $H$ of degree $cD^{e'}$, for some $e' = e'(\vphi) > 0$ and $c = c(\vphi,K)$, that contains all periodic points of degree $< D$. We write $H = \bigcup_{i=1}^{n} C_i$ where each $C_i$ is an irreducible curve over $\bb{C}$. 
\par 
Let $C_j$ be one of these curves that is not preperiodic. We will bound the number of periodic points of degree $< D$ on it. To do so, we consider the sequence of iterates $\{f^i(C_j)\}_{i \geq 1}$ and let $k$ be the smallest integer such that $f^k(C_j)$ is not one of the our curves $C_i$. Since each $C_i$ has degree at most $cD^{e'}$, it follows that $f^{k-1}(C_j)$ has degree at most $cD^{e'}$ and so $f^k(C_j)$ has degree at most $(cd)D^{e'}$. 
\par 
Now if $C_j$ has $N$ periodic points with degree $< D$, it follows that $f^k(C_j)$ also has $N$ periodic points with degree $< D$. These $N$ periodic points must necessarily also lie inside $H$ as they have degree $< D$ over $K$. Hence 
$$N \leq |f^k(C_j) \cap H| \leq (cd)D^{e'} \cdot c D^{e'} = c^2 d D^{2e'}.$$
\par 
If $C_j$ were instead preperiodic, we let $C'_1,\ldots,C'_m$ be the finitely many the periodic curves of $\vphi$. Then there exists some finite extension $L$ for which each $C'_i$ is defined over and for which there exists some positive integer such that $\vphi^{n_i}:C'_i \to C'_i$ is an endomorphism of $C'_i$ defined over $L$. Since $\vphi$ is polarized on $\bb{P}^2$, it must be that $\vphi^{n_i}$ restricted to $C'_i$ is a polarized endomorphism and thus by possibly passing to yet another finite extension, we may assume that $C'_i$ is birational to an elliptic curve or $\bb{P}^1$ and $\vphi^{n_i}$ is an endomorphism of degree $\geq 2$. For elliptic curves, it is then well known that the number of periodic points of degree $D$ on each $C'_i$ may be bounded by $O_{\vphi,K}(D^4)$ and for $\bb{P}^1$, it follows by Theorem \ref{Galois1} that we can find $e_i > 0$ such that the number of periodic points of degree $D$ is at most $O_{\vphi,K}(D^{e_i})$.
\par 
In conclusion, for each irreducible $C_i$ of our hypersurface $H$, we can find some uniform $f = f(\vphi) > 0$ such that the number of periodic points of degree $\leq D$ on $C_i$ is at most $O_{\vphi}(D^f)$. Summing up over each $C_i$, we get at most $O_{\vphi,K}(D^{f+e'})$ periodic points of degree $D$ over $K$, with $f$ and $e'$ only depending on $\vphi$ as desired. 
\end{proof}

In general, it is not known which endomorphisms $\vphi: \bb{P}^2 \to \bb{P}^2$ have finitely many periodic curves. However in the case where $\vphi$ restricts to a polynomial endomorphism $\vphi: \bb{A}^2 \to \bb{A}^2$, there are some results of Xie \cite{Xie24} that classify when this happens under an additional assumption. Let $\ovl{\vphi}$ be the induced map on the line at infinity $L_{\infty} = \bb{P}^2 \setminus \bb{A}^2$. We say that $\ovl{\vphi}$ is not of polynomial type if $\ovl{\vphi}^{-n}(x) \to \infty$ for all $x \in L_{\infty}$. 

\begin{theorem}[Theorem 1.20 \cite{Xie24}] \label{TheoremXie1}
Let $\vphi: \bb{P}^2 \to \bb{P}^2$ be an endomorphism of degree $d \geq 2$ that restricts to an endomorphism of $\bb{A}^2$. Assume that the induced morphism $\ovl{\vphi}$ on the hyperplane at infinity $L_{\infty}$ is not of polynomial type. Then if $\vphi$ has infinitely many periodic curves, $\vphi$ must be of the form $(f(x,y),g(x,y))$ for homogeneous polynomials $f,g$ after conjugating by a translation. Furthermore, all but finitely many periodic curves are a line of the form $\{z_1 = \lambda z_2\}$.     
\end{theorem}

Under the assumptions of Theorem \ref{TheoremXie1}, we are able to prove cases of Theorem \ref{Galois2} without the assumption of finitely many periodic curves.
\par 
First, note that if $\vphi$ is of the form $(f(x,y),g(x,y))$ for homogeneous polynomials $f,g$, then $(0,0)$ is an exceptional point for $\vphi$. In particular if we conjugated by the translation $(x-a,y-b)$, then $(a,b)$ is an exceptional point for $\vphi$ and hence must lie in $\ovl{K}$. Thus after conjugating and possibly extending our base field, we may assume that $\vphi$ is of the form $(f(x,y),g(x,y))$ for $f,g$ homogeneous polynomials of degree $d$ and with coefficients in $K$. 

\begin{lemma} \label{HomogeneousGalois1}
Let $H_{\lambda} = \{z_1 = \lambda z_2\}$ be a periodic hyperplane. If $\tilde{\vphi}(z) = \frac{f(z,1)}{g(z,1)}$ is the rational function corresponding to $(f,g)$, then $\lambda$ must be a periodic point of $\vphi$. 
\end{lemma}

\begin{proof}
This is clear as $\vphi$ sends $H_{\lambda}$ to $H_{\tilde{\vphi}(\lambda)}$.     
\end{proof}

Let $\vphi^n(x,y) = (f_n(x,y),g_n(x,y))$ where $f_n,g_n$ are homogeneous polynomials of degree $d^n$ and let $\lambda$ be a fixed point of order $n$ where we assume that $n$ is minimal. Then $\vphi^n$ sends $H_{\lambda}$ back to itself and any periodic point on $H_{\lambda}$ is necessarily a periodic point for $\vphi^n$. 

\begin{lemma} \label{HomogeneousGalois2}
Under the isomorphism $H_{\lambda} \simeq \bb{P}^1$ by sending $(\lambda z , z) \mapsto z$, our endomorphism $\vphi^n$ becomes $z \mapsto z^{d^n} g_n(\lambda,1)$.
\end{lemma}

\begin{proof}
Our endomorphism $\vphi^n$ sends $(\lambda z,z)$ to $(f_n(\lambda z, z), g_n(\lambda z,z))$ and hence under the isomorphism $H_{\lambda} \simeq \bb{P}^1$, our endomorphism is $z \mapsto g_n(\lambda z,z)$. Since $g_n(x,y)$ is homogeneous of degree $d^n,$ we have
$$g_n(\lambda z,z) = z^{d^n} g_n(\lambda,1)$$
as desired.
\end{proof}

Let $c_n = g_n(\lambda,1)$ and set $\vphi_{\lambda}(z) = z^{d^n} c_n$. Then 
$$\vphi_{\lambda}^2(z) = z^{d^{2n}} c_n^{d^n + 1}, \ldots, \vphi^m_{\lambda}(z) = z^{d^{mn}}c_n^{d^{(m-1)n} + d^{(m-2)n} + \cdots +1}.$$

\begin{lemma} \label{HomogenousGalois3}
Let $z$ be a periodic point for $\vphi_{\lambda}(z)$. Then $c_n z^{d^n - 1}$ is a root of unity.
\end{lemma}

\begin{proof}
If $z$ is of period $m$, then we have
$$z^{d^{mn}} c_n^{\frac{d^{mn}-1}{d^n -1}} = z \implies (z^{d^n -1} c_n)^{\frac{d^{mn}-1}{d-1}} = 1$$
and so $z^{d^n-1} c_n$ is necessarily a root of unity. 
\end{proof}

We are now able to bound the number of periodic points of degree at most $D$ that lie on $H_{\lambda}$. 

\begin{lemma} \label{HomogeneousGalois4}
For each $\lambda$, there is at most $O_K(D^6)$ many $z$'s such that $z$ is a periodic point for $\vphi_{\lambda}$ and $[K(z,\lambda):K] \leq D$. 
\end{lemma}

\begin{proof}
Let $\lambda$ be given and let $z$ be a periodic point for $\vphi_{\lambda}$. Then we know that $z^{d^n-1} c_n$ is a root of unity, $\omega$, where $c_n = g_n(\lambda,1)$. Since $z^{d^n-1} c_n$ lies in $K(\lambda,z)$, it follows that $[K(\omega):K] \leq D$. Hence there are at most $O_K(D^2)$ many such $\omega$'s as a root of unity of order $n \geq 6$ has at least $\sqrt{n}$ many Galois conjugates over $\bb{Q}$. 
\par 
For each such $\omega$, any two solutions of $z^{d^n-1}c_n = \omega$ differ by a root of unity $\omega'$. Now if these two solutions are of degree at most $D$ over $K$, then we necessarily have $[K(\omega'):K] \leq D^2$. Then there are necessarily at most $O_K(D^4)$ possible choices of $\omega'$. Hence we get a total of $O_K(D^6)$ many possible $z$'s such that $z$ is a periodic point of $\vphi_{\lambda}$. 
\end{proof}

With this, we can prove our bound on the number of periodic points of degree $\leq D$ for endomorphisms $\vphi: \bb{P}^2 \to \bb{P}^2$ that restrict to an endomorphism on $\bb{A}^2$. 

\begin{theorem} \label{Galois4}
Let $\vphi: \bb{P}^2 \to \bb{P}^2$ be an endomorphism of degree $d \geq 2$ defined over a number field $K$ that restricts to an endomorphism of $\bb{A}^2$. Assume further that the induced morphism $\ovl{\vphi}$ on the line at infinity $L_{\infty}$ is not of polynomial type. Then there exists $e = e(\vphi) > 0$ such that $|\Per_{\vphi,D}| \leq O_{\vphi,K}(D^e)$ for every $D \geq 1$. 
\end{theorem}

\begin{proof}
We perform the same strategy as in Theorem \ref{Galois2}. It then suffices to bound the number of periodic points of degree at most $D$ in $H_{\lambda}$ for $\lambda$ being a periodic point of $\tilde{\vphi}(z)$. Given such a periodic point $(\lambda z ,z)$ that lies on $H_{\lambda}$, we must have $[K(\lambda,z):K] \leq D$. Then by Lemma \ref{HomogeneousGalois4}, there are at most $O_K(D^6)$ many such periodic points $z$. The same argument as in Theorem \ref{Galois2} then allows us to conclude the existence of $e = e(\vphi) > 0$ such that $|\Per_{\vphi,D}| \leq O_{\vphi,K}(D^e)$ as desired.
\end{proof}

We may also consider the extension $K_n = K(\Per_n)$, where we adjoin all periodic points of period $n$ to our field $K$. Baker \cite{Bak06} has proven in dimension one that
$$[K_n:K] \geq c \frac{|\Per_n|}{\log |\Per_n|}$$
for some constant $c = c(\vphi,K) > 0$. This translates to an exponential growth in terms of $n$ for the degree $[K_n:K]$. We will first show that for endomorphisms $\vphi: \bb{P}^2 \to \bb{P}^2$ of degree $d \geq 2$, we also have an exponential growth for the degree $[K_n:K]$. 

\begin{theorem} \label{PeriodicGalois1}
Let $X$ be a smooth projective surface defined over a number field $K$ and let $\vphi: X \to X$ be a polarized endomorphism of degree $d \geq 2$. Then there exists $c = c(\vphi,K)$ such that $[K_n:K] \geq c^n$. In fact there exists $c' > 1$ such that  for $1- O(c'^{-n})$ proportion of points $x \in \Per_n$, we will have $[K(x):K] \geq c^n$.       
\end{theorem}

\begin{proof}
This follows immediately from Proposition \ref{PeriodicCurve3} and Theorem \ref{Galois1}. By Theorem \ref{Galois1}, we know that all preperiodic points of degree at most $c^n$ lie on a curve of degree at most $C(c^{en})$, where $e$ depends only on $\vphi$ and $C$ depends on $\vphi$ and $K$. By Proposition \ref{PeriodicCurve3}, this curve contains at most $C^2 c^{2en} d^n$ many periodic points of period $n$. Hence if we take $c, c' > 1$ such that $C^2 c^{2en} < (c'^{-1}d)^n$, then for $1 - O(c'^{-n})$ proportion of points $x \in \Per_n$ will not lie on our curve and so will necessarily have degree $\geq c^n$ as desired.     
\end{proof}

\printbibliography

\end{document}